\documentclass{amsart}

\usepackage{amsmath, amsthm, amsfonts,amssymb,amscd,mathrsfs}
\usepackage[centertags]{amsmath}
\usepackage{psfrag,graphicx}
\usepackage{multirow}
\usepackage[all]{xy}
\usepackage{remark}
\usepackage{indentfirst}

\usepackage{epic,eepic}
\usepackage{color}
\usepackage{ebezier}
\usepackage{enumerate}
\usepackage{mathrsfs}
\usepackage{graphicx}

\theoremstyle{plain}

\newtheorem{Thm}{Theorem}[section]
\newtheorem{Lem}[Thm]{Lemma}
\newtheorem{Conj}[Thm]{Conjecture}
\newtheorem{Prop}[Thm]{Proposition}
\newtheorem{Cor}[Thm]{Corollary}
\theoremstyle{plain}
\newremark{Def}[Thm]{Definition}
\newremark{example}[Thm]{Example}
\newremark{Rem}[Thm]{Remark}
\newremark{Construction}[Thm]{Construction}
\newcommand{\ie}{{\it i.e.}}

\newcommand{\divi}{\mathrm{{div}}}

\newtheorem{Que}[Thm] {Question}
\numberwithin{equation}{section}

\newcommand{\Alb}{\mathrm{Alb}} 
\newcommand{\Proj}{\mathrm{Proj}} 
 
\newcommand{\Spec}{\mathrm{Spec}} 
\newcommand{\Len}{\mathrm{length}}
\newcommand{\Sym}{\mathrm{Sym}} 
\newcommand{\tor}{\mathrm{tor}} 
\newcommand{\cO}{\mathcal{O}} 
\newcommand{\cL}{\mathcal{L}} 
\newcommand{\cA}{\mathcal{A}}
\newcommand{\cE}{\mathcal{E}}
\newcommand{\cM}{\mathcal{M}}

\begin{document}
\title{On algebraic surfaces of general type with negative $c_2$}\author{Yi GU}
\address{
Peking University\\
5, Yiheyuan Road,
Beijing, China
}
\email{pkuguyi2010@gmail.com}

\address{
Institut de Math\'ematiques de Bordeaux \\
Universit\'{e} de Bordeaux\\
351, Cours de la Lib\'eration \\
33405 Talence, France 
}
\email{Yi.Gu@math.u-bordeaux1.fr} 
\begin{abstract}
We prove that for any prime number $p\ge 3$, there
  exists a positive number $\kappa_p$ such that $\chi(\cO_X)\ge \kappa_p
c_1^2$ holds true for all algebraic surfaces $X$ of general type in
characteristic $p$. In particular, $\chi(\cO_X)>0$. This answers a question of
 N. Shepherd-Barron when $p\ge 3$. 
\end{abstract}
\maketitle

\section{Introduction}
The Enriques-Kodaira
classification of algebraic surfaces divides proper smooth algebraic surfaces into four classes according to their Kodaira dimension $-\infty, 0,1,2$.  A lot of problems remain unsolved for the last class, the so-called \emph{surfaces of general type}. One of the leading problems among these is the following so-called geography problem of minimal surfaces of general type (see \cite{P}). 
\begin{Que}
Which values of $(a, b)\in \mathbb{Z}^2$ 
are the Chern invariants $(c_1^2,c_2)$ of a minimal
surface of general type ?
\end{Que}
Over the complex numbers, though not yet settled completely, much is known
about this problem. Here we collect some classical
relations between $c_1^2$ and $c_2$ of a minimal
surface $X$ of general type: 
\begin{enumerate}
\item[] \ \ \ \ \ \ \ \  \  \ \ \ \ \ \ \ \ \ \ \ \ \ \ \ \ \ \ \ \ \ \ \ \ \ \ \ \  $c_1^2>0$;
\item[] \ \ \ \ \ \ \ \  \  \ \ \ \ \ \ \ \ \ \ \ \ \ \ \ \ \ \ \  \ \ \ $c_1^2+c_2\equiv 0 \mod 12$;
\item[(N)] \ \ \ \ \ \ \ \ \ \ \ \ \ \ \ \ \ \ \ \ \ \ \   $5c_1^2-c_2+36\ge 0$;
\item[(BMY)] \ \ \ \ \ \ \ \  \  \ \ \ \ \ \ \ \ \ \ \ \ \ \ \ \ \ \ \ \ \ \ \ \ \ \ \  $3c_2\ge c_1^2$.
\end{enumerate}
The first inequality is from the definition of a minimal surface of general type, the second condition is from Noether's formula
\begin{equation}\label{Noether formula}
12\chi(\cO_X)=c_1^2+c_2;
\end{equation} the inequality (N) is derived from the
following so-called Noether's inequality
\begin{equation}\label{noether's inequality}
K_X^2\ge 2p_g-4,
\end{equation}here $p_g:=h^0(X,K_X)$. The last inequality (BMY) is called the Bogomolov-Miyaoka-Yau inequality. Due to (\ref{Noether formula}), the inequality (BMY) can also be interpreted as below
\begin{enumerate}
\item[(BMY)'] \ \ \ \ \ \ \ \  \  \ \ \ \ \ \ \ \ \ \ \ \ \ \ \ \ \ \ \ \ \ \ \ \  $9\chi(\cO_X)\ge c_1^2$.
\end{enumerate}
It is known that most of the numbers $(a,b)$ satisfying the above relations are the Chern numbers of a surfaces of general type over $\mathbb{C}$. For more details and backgrounds on these inequalities,  confer \cite{Mi}, \cite{Y}, \cite{BPV} Chap. 7, and \cite{I-S} Chap. 8 \& 9.

Then we turn to the geography problem in positive characteristic cases.  Noether's inequality $(\ref{noether's inequality})$ (see \cite{L}) and Noether's formula $(\ref{Noether formula})$ (see \cite{Ba} Chap. 5) remains true, while Bogomolov-Miyaoka-Yau inequality (BMY) as stated no long holds (\cite{SZ1}, \S~3.4). In fact, even the following weaker inequality (CdF) due to  Castelnuovo and de Franchis fails. 
\begin{enumerate}
\item[(CdF)] \ \ \ \ \ \ \ \  \  \ \ \ \ \ \ \ \ \ \ \ \ \ \ \ \ \ \ \ \ \ \ \ \ \ \ \ $c_2\ge 0$
\end{enumerate}(see {\it e.g.} Section 3 of this paper). So it is
natural to formulate an inequality in positive characteristics
bounding $c_2$ from below by $c_1^2$. Using Noether's formula, it is the same as bounding $\chi$ from below. In fact, N. Shepherd-Barron has already consider a similar question and  proved 
that  $\chi> 0$ (equivalently, $c_2>
-c_1^2$)  with  a few possible exceptional cases when $p\le 7$
(\cite{SB2}, Theorem~8).  Here we generalize it to the following question.
\begin{Que}\label{Question}
What is the optimal number $\kappa_p$ such that $\chi\ge \kappa_p c_1^2$ holds for all surfaces of general type defined over a field of characteristic $p$ ?
\end{Que}
By definition we have
\begin{equation*}
\begin{split}
\kappa_p=\inf\{& \chi/{c^2_1} \ | \ \text{minimal algebraic surface of general type defined} \\
&\text{over an algebraically closed field of characteristic } p\}.
\end{split}
\end{equation*}
In particular, $\kappa_p> 0$ implies $\chi>0$. 

The purpose of this paper is an investigation of $\kappa_p$. Of
course, it will be in the best situation if we can work out  $\kappa_p$ for each $p$, however this looks difficult and instead, we try to find some interesting bounds of $\kappa_p$, say, to show $\kappa_p>0$ for all $p>2$.  The main result of this paper is the following theorem.

\begin{Thm}[Main Theorem]\label{Main 2} Let $\kappa_p$ be defined as above, then 
\begin{enumerate}
\item if $p>2$, $\kappa_p>0$;
\item if $p\ge 7$, $\kappa_p> (p-7)/12(p-3)$;
\item $\lim\limits_{p\rightarrow \infty}\kappa_p=1/12$;
\item $\kappa_5=1/32$.
\end{enumerate}
\end{Thm} 
Moreover, we have a conjecture on the values of $\kappa_p$:
\begin{Conj}\label{Conjecture}
If $p\ge 5$, then $\kappa_p=(p^2-4p-1)/4(3p^2-8p-3)$.
\end{Conj}
Note that if $p=5$, then $$(p^2-4p-1)/4(3p^2-8p-3)=1/32,$$ and if
$p\ge 7$, then $$(p^2-4p-1)/4(3p^2-8p-3)>(p-7)/12(p-3).$$ This
conjecture comes from the computation of the numerical invariants of
Raynaud's examples in \cite{R} (see Subsection~\ref{Raynaud's
  example}). Another computation for a special kind of surfaces of
general type is also carried out in the last section of this paper  giving
  some evidence in favor of this conjecture. 

In \cite{SB2}, remark after Lemma 9,
  N. Shepherd-Barron raised the question whether any minimal surface
of general type $X$ satisfies $\chi(\cO_X)>0$. 
Our Theorem~\ref{Main 2} implies that the answer is yes if 
$p\neq 2$ :

\begin{Cor}
If $p\neq 2$, then $\chi>0$ holds for all surfaces of general type.
\end{Cor}
This corollary can  help to improve and to better understand several other results ({\it e.g.} \cite{BBT}, Proposition~2.2, \cite{SB}, Theorem 25, 26, \& 27) where the authors need to take care of the possibility of $\chi\le 0$. 

As another application of Theorem~\ref{Main 2}, we give the following theorem  concerning the canonical map of surfaces of general type, which can be seen as an analogue of A. Beauville's relevant result over $\mathbb{C}$ (\cite{Be}, Prop. 4.1, 9.1).
\begin{Thm}
Let $S$ be a proper smooth of surface of general type over an algebraically closed field of characteristic $p>0$ with $p_g(S)\ge 2$,   
\begin{enumerate}
\item if $p\ge 3$ and $|K_S|$ is composed with a pencil of curves of genus $g$, then we have $$g\le 1+\frac{p_g+2}{2\kappa_p(p_g-1)};$$
\item if $p\ge 3$ and the canonical map is a generically finite morphism of degree $d$, then we have $$d\le \frac{p_g+1}{\kappa_p(p_g-2)}.$$
\end{enumerate}
\end{Thm}
The proof of this theorem is a naive copy of Beauville's,  
replacing simply  the inequality (BMY)' there by $\chi\ge \kappa_p c_1^2$,
hence it will not be included in this paper. The interesting part of this theorem is the following remark. 
\begin{Rem}
If we bound $\chi(\cO_S)$ from below (hence it bounds $p_g\ge
\chi(\cO_S)+1$ from below) as Beauville did in \cite{Be} and
substitute $\kappa_p$ by our lower bounds given in Theorem~\ref{Main
  2}, we can bound $g$ and $d$ from above as in \cite{Be}. 
As far as I know,  whether Beauville's bounds on $g$ and $d$ are optimal is not yet solved, not to mention ours.
\end{Rem}

We shall briefly explain our idea. Note that once we know that the inequality (CdF) fails in positive characteristics, we immediately obtain $\kappa_p< 1/12$
from $(\ref{Noether formula})$ and moreover, in order to study $\kappa_p$ we only have to consider those surfaces of general type with negative $c_2$. The main ingredient of this paper is an elaborate study of the numerical invariants of algebraic surfaces of general type with negative $c_2$ after \cite{SB2}.

This paper is organized as follows. 

In Section 2, we give some necessary preliminaries.  We rewrite Tate's formula on genus change to obtain some intermediate results which is more or less implicit in both Tate's original paper \cite{Ta} and \cite{SC}. Then we recall the theory of flat double covers, a Bertini type theorem and some other supplements.

In Section 3, we give some examples of algebraic surfaces of general type with negative $c_2$ and compute some of their numerical invariants.

In Section 4, we study the numerical properties of surfaces of general type with negative $c_2$, and prove our Theorem \ref{Main 2} except for the equation $\kappa_5=1/32$. 

In Section 5, we carry out a calculation of a special kind of algebraic surfaces of general type with negative $c_2$, namely those $X$ whose Albanese fibration is hyperelliptic and has the smallest possible genus. We show that our conjectural $\kappa_p$ (Conjecture \ref{Conjecture}) are  the best bounds of $\chi/c_1^2$ for these surfaces. This also completes the proof of our main theorem. During the calculation, a lemma  on a special kind of singularities is used, as the proof is a bit long, we put it as an appendix afterwards this section .

\begin{Rem}
In this paper we shall use the following notation.
\begin{enumerate}
\item For any invertible sheaf $\mathcal{E}$ over a scheme, $\mathbb{P}(\mathcal{E}):=\Proj (\Sym (\mathcal{E}))$.  

\item If $S\to T$ is a morphism of schemes in 
characteristic $p$, we denote by $F_S : S\to S$ the absolute 
Frobenius morphism and by $F_{S/T} : S\to S^{(p)}$ 
the relative Frobenius morphism (where $S^{(p)}=S\times_T T$ is
obtained by base changing $S\to T$ via $F_T : T\to T$). If $\pi: S\to Y$ is
a morphism of $T$-schemes, we denote by $\pi^{(p)}: S^{(p)}\to Y^{(p)}$
the morphism of $T$-schemes induced by $\pi$. 
\end{enumerate}
\end{Rem}

\section{Preliminaries}
\subsection{Genus change formula} \label{subsection: genus change}
Let $S$ be a normal projective and geometrically integral curve over a field $K$ (in particular $H^0(S,\cO_S)=K$) 
of positive characteristic $p$, of arithmetic genus 
$g(S):=1-\chi(\cO_S)=\dim H^1(S, \cO_S)$. 
The latter is also called the genus of the function field $K(S)$. 
Let $L/K$ be a finite extension and let $(S_L)'$ be the normalisation of 
$S_L:=S\times_K L$. A theorem of Tate (\cite{Ta}) states that 
$$(p-1) \mid 2(g((S_L)')-g(S)).$$ 
This is proved in the scheme-theoretical language in \cite{SC}. Below 
we give a slightly different proof in the scheme-theoretical language (in some places close to Tate's original one) 
and some more precise intermediate results, 
in particular, we show that if $g(S)$ is small with respect to $p$, then
the normalisation of $S^{(p)}$ is smooth (Corollary~\ref{complete}). 

\begin{Lem}\label{exact1} 
Let $S, Y$ be geometrically integral normal curves over a field $K$
of positive characteristic $p$, let $\pi: S\rightarrow Y$ be a
finite inseparable morphism of degree $p$. Then 
$\Omega_{S/Y}$ is invertible and we have an exact sequence 
  \begin{equation}
    \label{eq:omSY}
    0 \to F_S^*\Omega_{S/Y} \to \pi^*\Omega_{Y/K}\to
      \Omega_{S/K} \to \Omega_{S/Y}\to 0
  \end{equation}
with $F_S^*\Omega_{S/Y}\simeq \Omega_{S/Y}^{\otimes p}$. 
\end{Lem}

\begin{proof} The second part $\pi^{*}\Omega_{Y/K}\to \Omega_{S/K}\rightarrow
\Omega_{S/Y}\to 0$ is canonical and always exact.  Let us show the existence of 
a complex $0\to F_S^*\Omega_{S/Y} \to \pi^*\Omega_{Y/K}\to \Omega_{S/K}$ 
and prove the exactness under the assumption of the lemma. 

As $\pi$ is purely inseparable of degree $p$, we have
the inclusions of functions fields $K(S)^p \subseteq K(Y)\subseteq
K(S)$, hence $F_{S/K} : S\to S^{(p)}$ factors through $\pi : S\to Y$ 
and some $f: Y\to S^{(p)}$ (which is in fact the normalisation map). 
We have a canonical commutative diagram 
$$
\xymatrix{
S \ar[r]_{F_{S/K}} \ar@/^/[rr]^{F_S} \ar[d]_{\pi} &  S^{(p)}\ar[d]^{\pi^{(p)}} 
\ar[r]_{q} & S \ar[d]^\pi \\ 
Y  \ar[r] \ar[ur]_{f}    &  Y^{(p)} \ar[r] & Y}
$$
where $q: S^{(p)}\to S$ is the projection map. We have 
$q^*\Omega_{S/Y}=\Omega_{S^{(p)}/Y^{(p)}}$ because the last square is \emph{Cartesian},
and a canonical map $f^*\Omega_{S^{(p)}/Y^{(p)}}\to \Omega_{Y/Y^{(p)}}$,
hence a canonical map
$F_S^*\Omega_{S/Y}=\pi^*f^*\Omega_{S^{(p)}/Y^{(p)}}\to \pi^*\Omega_{Y/Y^{(p)}}$. 
Note that the canonical map 
$F_{Y/K}^*\Omega_{Y^{(p)}/K}\to \Omega_{Y/K}$ is identically zero, so
the canonical map $\Omega_{Y/Y^{(p)}}\to\Omega_{Y/K}$ is an isomorphism. 
Therefore we have a map 
$\Phi : F_S^{*}\Omega_{S/Y}\to \pi^{*}\Omega_{Y/K}$. Its composition 
with $\pi^*\Omega_{Y/K}\to \Omega_{S/K}$ is zero because locally it 
maps a differential form $db$ to $d(b^p)=0$. So 
$$0\to F_{S}^{*}\Omega_{S/Y}\to \pi^{*}\Omega_{Y/K}\to \Omega_{S/K}$$ 
is a complex. 

Let $s\in S$ and let $y=\pi(s)\in Y$. 
Then $A:=\cO_{Y,y}\to B:=\cO_{S,s}$ is a finite extension of discrete
valuation rings of degree $p$, so $B=A[T]/(T^p-a)$ for
some $a\in A$ (the element $a\in A$ is either a uniformizing element or
a unit whose class in the residue field of $A$ is a not a $p$-th power). 
The stalk of the complex \eqref{eq:omSY} becomes 
\begin{equation} \label{eq:local} 
0\to Bda \to \Omega_{A/K}\otimes_A B\to 
((\Omega_{A/K}\otimes_A B)\oplus BdT)/Bda \to BdT \to 0
\end{equation} 
which is clearly exact. This also shows that $\Omega_{S/Y}$ is locally free
of rank $1$. As a general fact, we then have 
$F_{S}^*\Omega_{S/Y}\simeq \Omega_{S/Y}^{\otimes p}$.
\end{proof}

\begin{Prop}\label{genus-c} Let $S, Y$ be normal projective geometrically 
integral curves over $K$ and let $\pi : S\to Y$ be a finite inseparable 
morphism of degree $p$. 
Let $\mathcal A=\mathrm{Ker}(\Omega_{S/K}\to \Omega_{S/Y})$. Then 
\begin{enumerate}
\item $\mathcal A=\Omega_{S/K,\tor}$ the torsion part of 
$\Omega_{S/K}$ and we have 
$$(p-1)\deg\det(\mathcal A)=2p(g(S)-g(Y));$$   
\item $\deg\det(\mathcal A)=\deg\mathcal A=\sum_{s\in S} (\Len_{\cO_{S,s}}\mathcal A_s)
[K(s) : K]=\dim_{K}H^{0}(S,\mathcal A)$; 
\item if $g(S)=g(Y)$, then $S$ is smooth over $K$.
\end{enumerate}
\end{Prop}

\begin{proof} (1) Split the exact sequence \eqref{eq:omSY} into 
\begin{equation} \label{eq:split} 
0\to \Omega_{S/Y}^{\otimes p} \to 
\pi^*\Omega_{Y/S}\to \mathcal A\to 0
\end{equation}
and 
$$0\to \mathcal A\to \Omega_{S/K}\to \Omega_{S/Y}\to 0.$$
As $\Omega_{S/K}$ has rank $1$ (because $S$ is geometrically reduced) 
and $\Omega_{S/Y}$ is invertible, we 
have  $\mathcal A=\Omega_{S/K, \tor}$. 
As \emph{$\det\Omega_{S/K}=\omega_{S/K}$} and similarly for $\Omega_{Y/K}$, 
by taking the determinants in the two exact sequences we get
$$\pi^*\omega_{Y/K}=\det\mathcal{A}\otimes \omega_{S/Y}^{\otimes p},$$ and $$\omega_{S/K}=\det\mathcal{A}\otimes \omega_{S/Y}.$$ Hence
$$(\det\mathcal A)^{\otimes (p-1)}\simeq \omega_{S/K}^{\otimes p}
\otimes \pi^*\omega_{Y/K}^{-1}.$$
By Riemann-Roch, $\deg\omega_{S/K}=2g(S)-2$ (and similarly for
$Y$). Part (1) is then obtained by taking the degrees in the above 
isomorphism. 

(2) This is well known and can be proved locally at every stalk of
$\mathcal A$ (see {\it e.g.}, \cite{LLR}, Lemma 5.3(b)). 

(3) Finally, if $g(S)=g(Y)$, then $\deg \mathcal A=0$, hence 
$\mathcal A=0$. This implies that $\Omega_{S/K}$ free of rank one, 
hence $S$ is smooth over $K$. 
\end{proof}

\begin{Rem}\label{modulo p}
The support of $\mathcal{A}$ consists of singular (more precisely speaking, non-smooth) points
of $S$, and it is well known that such points are inseparable over $K$ 
(\cite{LQ}, Proposition 4.3.30). In particular $p \mid [K(s):K]$ for any $s\in \mathrm{Supp}(\mathcal{A})$.
\end{Rem}

\begin{Cor}\label{Tate} {\rm (Tate genus change formula)} Let $S$ be a normal 
projective geometrically integral curve over $K$. Let $L$ be an 
algebraic extension of $K$ and let $Y$ be the normalisation of $S_L$
(viewed as a curve over $L$). Then 
 $$p-1 \mid 2(g(S)-g(Y)).$$ 
\end{Cor}

\begin{proof} 
The result will be derived from Proposition~\ref{genus-c}. 
We can suppose $L/K$ is purely inseparable. Let us 
first treat the case $L=K^{1/p}$. Decompose the absolute Frobenius
$K\to K$, $x\mapsto x^p$ as 
$$K\stackrel{i}{\to} K^{1/p} \stackrel{\rho}{\to} K$$
where $i$ is the canonical inclusion and $\rho$ is an isomorphism. 
Let us extend $Y$ to $Y_K:=Y\otimes_{L}K$ using $\rho$. Then $Y_K$
is a normal projective and geometrically integral curve over $K$,
of arithmetic genus (over $K$) equal to that of $Y$ over $L$. Moreover
$Y_K$ is birational to $(S_L)\otimes_L K=S^{(p)}$. So we have 
an inseparable finite morphism $S\to Y_K$ of degree $p$. By 
Proposition~\ref{genus-c}(1), we have $(p-1) \mid 2(g(S)-g(Y))$
and $g(S)>g(Y)\ge 0$
unless $S$ is already smooth over $K$. Repeating the same argument, 
for any $n\ge 1$, if $S_n$ denotes the normalisation of $S_{K^{1/p^n}}$, 
then $p-1$ divides $2(g(S)-g(S_n))$, and $S_n$ is smooth over $K^{1/p^n}$ if
$n$ is big enough. 

Now let $L/K$ be a finite purely inseparable extension. Then 
there exists $m\ge 1$ such that $L\subseteq K^{1/p^m}\subseteq L^{1/p^m}$
and $S_m$ is smooth. This implies that the normalisation 
$Y_m$ of $Y_{L^{1/p^{m}}}$ is $(S_m)_{L^{1/p^m}}$. On the other hand,
applying the previous result to the $L$-curve $Y$ instead of 
$S$, we see that $p-1$ divides $2(g(Y)-g(Y_m))$. As 
$g(Y_m)=g(S_m)$, we find that $p-1$ divides $2(g(S)-g(Y))$. 
The case of any algebraic extension follows immediately. 
\end{proof}

\begin{Lem}\label{torsion} Let $\pi: S\to  Y$ be as in Lemma~\ref{exact1}. 
\begin{enumerate}
\item We have 
$$p\deg\Omega_{Y/K, \tor} \le \deg \Omega_{S/K, \tor}.$$ 
\item Let $s\in S$ and let $y=\pi(s)$. Suppose that $K(y)=K(s)$, then 
$$\Len_{\cO_{S,s}}(\Omega_{S/K, \tor})_s\ge p\dim_{K(s)}\Omega_{K(s)/K}. $$ 
\end{enumerate}
\end{Lem}

\begin{proof} (1) Denote $\mathcal A= \Omega_{S/K, \tor}$ and $\mathcal B= \Omega_{Y/K, \tor}$. Let $s\in S$ and $y=\pi(s)$. 
The canonical map 
$\mathcal B_y\otimes\cO_{S,s}=\pi^*(\mathcal B)_s \to \mathcal A_s$ is injective 
by the exact sequence \eqref{eq:split}, because 
$\Omega_{S/Y}^{\otimes p}$ is torsion-free. Therefore 
$$e_{s}\Len_{\cO_{Y,y}}(\mathcal B_y)=\Len_{\cO_{S,s}}(\mathcal B_y\otimes \cO_{S,s})\le \Len_{\cO_{S,s}}(\mathcal A_s)$$ 
where $e_s$ is the ramification index of $\cO_{Y,y}\to \cO_{S,s}$. The desired inequality 
holds because $p=e_s[K(s):K(y)]$. 

(2) Let $A=\cO_{Y,y}$, $B=\cO_{S,s}$. As $K(y)=K(s)$, $A\to B$ has ramification 
index $p$. So $B=A[T]/(T^p-t)$ for some uniformizing element $t$ of $A$. 
The exact sequence \eqref{eq:local} gives the exact sequence 
$$ 0\to (\Omega_{A/K}/Adt) \otimes_A B \to \Omega_{B/K}=((\Omega_{A/K}/Adt) \otimes_A B) \oplus BdT.$$  
In particular, 
\begin{equation}\label{eq:torsion2} 
\mathcal A_s = (\Omega_{A/K}/Adt) \otimes_A B.
\end{equation}
The usual exact sequence 
$$ tA/t^2A \to \Omega_{A/K}\otimes_A K(y) \to \Omega_{K(y)/K}\to 0,$$ 
implies we have a surjective map
$$ \mathcal A_s\twoheadrightarrow \Omega_{K(y)/K}\otimes_A
B=\Omega_{K(y)/K}\otimes_{K(y)} B/tB.$$ 
So 
$$\Len_{B}\mathcal A_s\ge p\dim_{K(y)}\Omega_{K(y)/K}=p
\dim_{K(s)}\Omega_{K(s)/K}.$$ 
\end{proof}

\begin{Cor} \label{remark on decreasing by p} Let $S=S_0\to S_1 \to \cdots \to S_n$ be a tower of 
inseparable covers of degree $p$ of geometrically integral 
normal projective  curves over $K$. Let $g_i=p_a(S_i)$. Then 
$g_{i+1}-g_i \le (g_i-g_{i-1})/p.$ In particular $\deg \Omega_{S/K,\tor}=2p(g_0-g_1)/(p-1)>2(g_{0}-g_{n})$ by Proposition \ref{genus-c}.\end{Cor}

Lemma~\ref{torsion}(2) is not used in the sequel. But we think 
it can be of some interest in the understanding of genus
changes. It implies immediately \cite{Sal}, Corollary 3.3.
\begin{Def}\label{g-r}
We call a curve $S$ over $K$ 
\emph{geometrically rational} if $S_{\bar{K}}$ is integral with 
normalisation isomorphic to $\mathbb P^1_{\bar{K}}$.
\end{Def}


A slightly weaker version of the next corollary can also be found in \cite{Sal}, Corollary 3.2.

\begin{Cor} \label{complete}
Let $S$ be a projective normal and geometrically rational curve over $K$ of
(arithmetic) genus $g$. Suppose that $S$ is not smooth. 
Let $Y$ be the normalisation of $S^{(p)}$. 
\begin{enumerate}
\item We have $2g\geq (p-1)$. If $2g= p-1$, then $Y$ is a smooth conic, 
Moreover, $S$ has exactly one non-smooth point, the latter being 
of degree $p$ over $K$.
\item If $g< (p^2-1)/2$, then $Y$ is a smooth conic over $K$. In particular, we have  $\deg \Omega_{S/K, \tor}=2pg/(p-1)$.
\end{enumerate}
\end{Cor}

\begin{proof} (1) This is an immediate consequence of 
Proposition~\ref{genus-c}. 

(2) If $Y$ is not smooth, as non-smooth points have inseparable residue fields
(see {\it e.g.} \cite{LQ}, Prop 4.3.30), we have 
$\deg\Omega_{S/K, \tor}\ge p\deg \Omega_{Y/K, \tor}\ge p^2$  by Lemma~\ref{torsion}(1). 
So $g\ge g(Y)+p(p-1)/2\ge (p^2-1)/2$  since $g(Y)\geq (p-1)/2$, contradiction. So $Y$ is smooth. In particular, $Y$ is a smooth conic because $S$ is assumed to be geometrically rational.
\end{proof}

\subsection{Flat double covers}\label{flat double cover}
\setlength{\parindent}{2em}We recall some basic facts on flat double covers. One can also consult 
\cite{C-D}, \S 0 or \cite{BPV}, III, \S 6-7  for a standard introduction. 

\begin{Def}
A finite morphism between noetherian schemes $f:S\rightarrow Y$ is called 
a \emph{flat double cover} if $f_{*}\mathcal{O}_{S}$ is locally free 
of rank $2$ over $\cO_Y$.
\end{Def}

For our purpose we suppose that $Y$ is an \emph{integral noetherian}
scheme defined over a field $K$ of characteristic
\emph{different from $2$} in this subsection. 

\begin{Construction}\label{Ls} 
Flat double covers of $Y$ can be constructed as 
follows. Choose an invertible sheaf $\mathcal L$ on $Y$, and choose  
$s\in H^{0}(Y, \mathcal L^{\otimes 2})=Hom_{\cO_Y}(\cL^{-2},\cO_Y)$. Endow the $\cO_Y$-module 
$\cO_Y\oplus \cL^{-1}$ with the $\mathcal{O}_{Y}$-algebra structure 
by 
$$\cL^{-1}\times \cL^{-1}\to \cL^{\otimes (-2)} \stackrel{e(s)}\to \cO_Y$$
where $e(s)$ is the evaluation at $s$. 
Then $S:=\mathrm{Spec}(\cO_Y\oplus \cL^{-1})$ is a flat double over of 
$Y$. Note that if we replace $s$ with $a^2s$ for some $a\in
H^0(Y,\cO_Y)^*$, then we get a flat double cover isomorphic to 
the initial one.  We call the invertible sheaf $\cL$ above as  \emph{the associated invertible sheaf} of $f$.

Conversely, if $f:S\rightarrow Y$ is a flat double
cover, we have a trace morphism: $f_*\cO_S\rightarrow \cO_Y$, since $p\neq 2$, this trace morphism splits $f_*\cO_S$ into direct sum $\cO_Y\oplus \cL^{-1}$, where $\cL^{-1}$ is the kernel of the trace morphism. Now it is clear that the $\cO_Y$-algebra structure of $f_*\cO_S$ is given by $\cL^{-1}\times \cL^{-1}\to \cO_Y$ as any elements in $\cL^{-1}$ has null trace. 
For the cover $S\to Y$ defined as above, if $s\ne 0$, 
$S$ is reduced and $S\to Y$ is generically \'etale, the \emph{branch divisor} 
is equal to $B:=\mathrm{div}(s)$.  
\end{Construction}
From this construction we immediately obtain the formula of dualizing sheaf.
\begin{Cor}\label{formula of dualizing sheaf}
We have $\omega_{S/Y}=f^*\cL$. \qed
\end{Cor}
\begin{Cor} \label{genus}  
  \begin{enumerate}
\item If $f: Y'\to Y$ is a morphism of integral noetherian schemes,  then 
$S\times_Y Y'\to Y'$ is a flat double cover obtained by 
$f^*\cL$ and $f^*s\in H^0(Y', (f^*\cL)^{\otimes 2})$. 
  \item If $Y$ is a geometrically connected smooth projective curve over $K$, and
    $S\rightarrow Y$ is a flat double cover with branch divisor $B$,
    then 
$$p_{a}(S)=2p_{a}(Y)-1+\deg (B)/2.$$
\item If $Y$ is a geometrically connected smooth projective surface over $K$, then:
$$\chi(\mathcal{O}_{S})=\chi(\mathcal{O}_{Y})+\chi(\mathcal{L}^{-1})=2\chi(\mathcal{O}_{Y})+(B^2+2B\cdot
K_{Y})/8$$
where $K_Y$ is the canonical divisor of $Y$. \qed
  \end{enumerate}
\end{Cor}

\begin{Prop}\label{Property}
Let $f: S\rightarrow Y$ be a flat double cover over $Y$ 
with branch divisor $B$. 
\begin{enumerate}
\item If $Y$ is normal, then $S$ is normal if and only if $B$ is reduced.  
\item If $Y$ is regular, then $S$ is regular if and only if $B$ is regular.
\item If $Y$ is smooth over $K$, then $S$ is smooth over $K$ if and only if $B$ is smooth over $K$.\end{enumerate}
\end{Prop}
\begin{proof}
See \cite{C-D}, chapter $0$. 
\end{proof}

Now suppose $Y$ is regular. Let $f: S\to Y$ be a flat 
double cover given by $\cL$ and $s\ne 0$ as in \ref{Ls} with 
$B=\mathrm{div}(s)$ being the branch divisor of $f$. 
Then $B$ can be uniquely written as sum of effective
divisors: $B=B_{1}+2B_{0}$ such that $B_{1}$ is reduced.

\begin{Prop}\label{normalisation} The normalisation of 
$S$ is the flat double cover  $S'\to Y$ given by 
$\cL':=\cL\otimes \cO_Y(-B_{0})$ and 
$s\in H^{0}(Y, {\cL'}^{\otimes 2})$ (here we use $\cL^{\otimes 2}=s\cO_Y(B)\supseteq s\cO_Y(B_1)=\cL'^{\otimes 2}$, and $s$ is in fact also an global section in $\cL'^{\otimes 2}$). 
Moreover, $B_{1}$ is the branch divisor of $S'\to Y$.
\end{Prop}

As an application of this proposition, we recall the following process 
of resolution of singularities from a flat double cover. One may also confer \cite{BPV}, III $\S~6$. 

\begin{Def}[Canonical resolution]\label{canonical}
Let $k$ be an algebraically closed field of characteristic different from $2$, and for our purpose, let $Y_0$ be either a nonsingular algebraic surface over $k$, or the spectrum of a local ring of a nonsingular algebraic surface over $k$. Let $f_0:S_{0}\rightarrow Y_0$ be a flat double cover given by data 
 $\{\mathcal{L}_0, 0\neq s\in H^{0}(Y_0,\mathcal{L}_0^{\otimes 2})\}$ and assume that the branch locus $B:=\divi(s)$ is 
reduced (\ie \ $S_{0}$ is normal by Proposition~\ref{Property}). 
Then the canonical resolution of singularities of $S_{0}$ is the following process:
 
If $B_{0}$ is not regular, choose a singular point 
$y_{0}\in B_{0}$, let $m_{0}:=\mathrm{mult}_{y_{0}}B_{0}$, and $l_0:=\lfloor m_0/2 \rfloor$. Blowing up $y_0$ we obtain a morphism $\sigma_0:Y_{1}\rightarrow Y_{0}$. 
Then $S_{0}\times_{Y_{0}}Y_{1}\to Y_1$ is a flat double cover with associated invertible sheaf $\cL'=\sigma^*\cL_0$
 and branch divisor $\sigma_0^*B=\widetilde{B}_0+m_0E$, where 
$\widetilde{B}_0$ is the strict transform of $B_0$ in $Y_1$ and $E$ 
is the exceptional divisor.
Let $S_{1}$ be the normalisation of $S_{0}\times_{Y_{0}}Y_{1}$. 
Then by Proposition \ref{normalisation}, 
$f_1 : S_{1}\to Y_1$ is a flat double cover with associated invertible sheaf $\cL_1=\sigma^*\cL_0\otimes\cO_{Y_1}(-l_0E)$ and branch divisor $B_{1}=\sigma^{*}(B_0)-2l_0E$. Replace our data $\{f_0, \cL_0, B_0\}$ by $\{f_1,\cL_1, B_1\}$ and run the above process again until we reach some $n$ such that $B_n$ is regular, \ie \ $S_n\to S_0$ is a resolution of singularities by Proposition~\ref{Property}. To see why this process stops in finitely many times, one may confer \cite{BPV} Chap. 3.7. We draw the following diagram as a picture of this process.\qed
\end{Def}
$$\xymatrix{
S_0 \ar[d]_{f_0}& S_1\ar[d]_{f_1}\ar[l]^{g_0}&...\ar[l]^{g_1}&S_n\ar[d]^{f_n}\ar[l]^{g_{n-1}}\ar@/_1pc/[lll]_{g}\\
Y_0 & Y_1\ar[l]^{\sigma_0}&...\ar[l]^{\sigma_1}&Y_n\ar[l]^{\sigma_{n-1}}}
$$

We will denote by $y_{i}\in B_{i}$ the center of the blowing-up morphism 
$\sigma_{i}:Y_{i+1}\to Y_i$, $E_{i}$  the exceptional locus, $m_{i}:=\mathrm{mult}_{y_{i}}B_{i}$, and 
$l_{i}:=\lfloor m_{i}/2\rfloor$. Then it follows that
\begin{align}\label{equation on R^1}
\chi(R^1{g_i}_*\cO_{S_{i+1}})&=(l_i^2-l_i)/2. \\
\omega_{S_{i+1}/Y_{i+1}}&=g_i^*\omega_{S_i/Y_i}\otimes f_i^*\cO_{Y_{i+1}}(-l_iE_i).
\end{align}
In particular, if $Y$ is proper, then \begin{equation}
\chi(\cO_{S_n})-\chi(\cO_{S_0})=-\sum\limits_{0\le i< n} (l_i^2-l_i)/2.
\end{equation}
\begin{equation}\label{variation of K^2}
K_{S_{i+1}}^2=K_{S_i}^2-2(l_i-1)^2.
\end{equation} 
\begin{Def}\label{num}
Given a flat double cover $f_0: S_0\to Y_0$ as above, and  assume $g:S_n\to S_0$ is the canonical resolution defined as above. Let $y$ be a closed point of the branch divisor $B$, then there is a unique $s\in S_0$ lying above $y$, we define $\xi_y:=\dim_k R^1g_*(\cO_{S_n})_s$. It is well known that if $g': \widetilde{S}\rightarrow S_0$ is another resolution of singularities, then $\xi_{y}=\mathrm{dim}_k (R^{1}g'_{*}\mathcal{O}_{\widetilde{S}})_{s}.$
\end{Def}

Keep the notations we introduced for the canonical resolution, then by formula (\ref{equation on R^1}) we can compute $\xi_y$:
 \begin{equation}\label{equation of xi 1}
\xi_y:=\sum_{i\le n-1} \delta_{i}(y)(l_{i}^2-l_{i})/2,
\end{equation} 
where
$$
\delta_i(y)=\left\lbrace
\begin{matrix}
1, & \text{if $y_i$ is mapped to $y$ by $Y_i\to Y$,} \\ 
0, & \text{otherwise.} \hfill \\ 
\end{matrix}
\right.
$$ 

By definition, in case $Y$ is  projective, we have :
\begin{equation}\label{difference}
\chi(\cO_{S_0})-\chi(\mathcal{O}_{S_{n}})=\chi(R^1g_*\cO_{S_n})=\sum\limits_{y\in B}\xi_{y}.
\end{equation} 

\begin{Def}\label{neglect}
\begin{enumerate}
\item A point $y\in B$ as above is called a negligible singularity of the first kind, if $B$ is locally the union of two nonsingular divisors.

\item  A point $y\in B$ as above is called a negligible singularity of the second kind, if $B$ is locally the union of three nonsingular divisors such that at least two of them meet properly at $y$.
\end{enumerate}

It is evident from (\ref{equation of xi 1}) both kinds of negligible singularities has $\xi_y=0$. So we are allowed to neglect them in the computation of $\chi(\cO_{X_n})$.
\end{Def}

Finally we have another application of Proposition~\ref{normalisation}.
\begin{Def}
In this paper we call a projective curve $E$ over a field $K$ is \emph{hyperelliptic} (resp. \emph{quasi-hyperelliptic}) 
 if  it is geometrically integral and admits 
a flat double cover over 
$\mathbb{P}_{K}^{1}$ (resp. a smooth plane conic).  
\end{Def}
\begin{Prop}\label{on Pa}
Let $E$ be a normal projective geometrically rational  curve (see Definition~\ref{g-r}) over a field $K$ of characteristic $p \neq 2$. 
If $E$ is quasi-hyperelliptic,  then $p_{a}(E)=(p^{i}+p^{j}-2)/2$ for some non-negative integer $i,j$.
\end{Prop}

\begin{proof} We can extend $K$ to its separable closure and suppose that $K$ is separably closed. 
We have a flat double cover $E\to \mathbb P_K^1$ with
  reduced branch divisor $B$ ($E$ is normal). Write $B=b_{1}+...+b_{n}$. 
Let $d_{i}:=[k(b_{i}):K]$,  
this is a power of $p$. The flat double cover 
$E_{\bar{K}}\to \mathbb P^1_{\bar{K}}$ has its branch divisor $B_{\bar{K}}$ supported in $n$
points, with multiplicities powers of $p$.  By Proposition~\ref{normalisation}, the normalisation 
of $E_{\bar{K}}$ is a flat double cover of $\mathbb P^1_{\bar{K}}$ branched at $n$ points. This normalisation being a smooth rational curve, we find $n=2$ by Corollary \ref{genus}(2). So $\mathrm{deg}(B)=d_{1}+d_2$ is of the form $p^{i}+p^{j}$ and $p_{a}(E)=\deg(B)/2-1$ is of form $(p^{i}+p^{j}-2)/2$ by Corollary~\ref{genus}(2).
\end{proof}


\subsection{On a Bertini type theorem}\label{Bertini}
\setlength{\parindent}{2em}Let $S$ be a proper scheme over a field $k$, and let $\mathcal{L}=\cO_S(D)$ be an invertible sheaf on $S$. By $|D|$ we denote the set of the effective divisors linearly equivalent to $D$. 
Let $H^0(S, \cO_S(D))^{\vee}$ be the dual of the $k$-vector space $H^0(S, \cO_S(D))$. 
We have a bijection 
$$(H^0(S, \cO_S(D))\setminus \{0\})/k^*=\mathbb P(H^0(S, \cO_S(D))^{\vee})(k)\to |D|$$ 
which maps $s\in H^0(S, \cO_S(D))\setminus \{0\}$ to $D+\mathrm{div}(s)$. 

A sub-linear system $V$ of $|D|$ is, by definition, the set of $D+\mathrm{div}(s)$, $s\in \widetilde{V}\setminus \{0\}$, 
where $\widetilde{V}$ is a linear subspace of $H^0(S, \cO_S(D))$, we call this linear system the associated linear system of $V$. The above
bijection establishes a bijection between $V$ and the rational 
points $\mathbb P(\widetilde{V}^{\vee})(k)$. 

Let $f: X\rightarrow C$ be a flat fibration between proper
integral varieties over an infinite field $k$.  Let
$K$ be the function field of $C$, and let $X_{\eta}/K$ denote the
generic fibre of $f$. Let $\mathcal{L}=\cO_X(D)$ be an invertible
sheaf on $X$, and let $V\subseteq |D|$ be a sub-linear system. 
Denote by $D_\eta$ the restriction of $D$ to $X_\eta$ and by $V_K$ 
the sub-linear system of $|D_{\eta}|$ 
generated by the effective divisors $D'_\eta$, $D'\in |D|$. The 
vector space $\widetilde{V}_K$ associated to $V_K$ 
is exactly $K(i(\widetilde{V}))\subseteq H^0(X_\eta, \cO_{X_{\eta}}(D_\eta))$, where 
$i: H^0(X, \cO_X(D))\hookrightarrow H^0(X_\eta, \cO_{X_\eta}(D_\eta))$ 
is the canonical restriction map. 

\begin{Lem}\label{Property P}
Consider the map 
$$r: V=\mathbb P(\widetilde{V}^{\vee})(k)\to V_K=
\mathbb P((\widetilde{V}_K)^{\vee})(K)$$ 
defined by $D'\mapsto D'_\eta$. Then 
$r$ is continuous for the Zariski topology. Moreover, 
for any Zariski non-empty open subset $U$ of $V_K$, $r^{-1}(U)$
is a non-empty Zariski open subset of $V$. 
\end{Lem}
\begin{proof}
Let $(\widetilde{V}_K)^{\vee}\hookrightarrow
\widetilde{V}^{\vee}\otimes_k K$ be the dual map of the surjective map
$\widetilde{V}\otimes_k K \to \widetilde{V}_K$. It 
induces a dominant rational map
$\mathbb{P}(\widetilde{V}^{\vee}\otimes_k K)\dashrightarrow
\mathbb{P}((\widetilde{V}_K)^{\vee})$. Let $\Omega$ be the domain of 
definition of this rational map. Then we see easily that the canonical 
map $\mathbb{P}(\widetilde{V}^{\vee})(k)\to \mathbb{P}(\widetilde{V}^{\vee})(K)$
is continuous for the Zariski topology, has image in $\Omega$ and the composition 
$\mathbb{P}(\widetilde{V}^{\vee})(k)\to
\mathbb{P}(\widetilde{V_K}^{\vee})(K)$ is equal to $r$. 

So $r$ is continuous for the Zariski topology. In particular $r^{-1}(U)$
is open. As $k$ and $K$ are infinite, it is well known that 
$\mathbb{P}(\widetilde{V}^{\vee})(k)\hookrightarrow
\mathbb{P}(\widetilde{V}^{\vee}\otimes_k K)(K)
=\mathbb{P}(\widetilde{V}^{\vee})(K)$ has 
dense image, and the latter is dense in 
$\mathbb{P}(\widetilde{V}^{\vee}\otimes_k K)$. So $r^{-1}(U)$ is non-empty. 
\end{proof}

We say that a general member of $V$ has a
certain property (P) if there is a non-empty (Zariski) open subset
of $\mathbb P(\widetilde{V})(k)$ such that each
member in this subset satisfies the property (P). This lemma then shows that if a general member of $V_K$ has property (P), so does $D'_\eta:=D‘|_\eta$, where $D'\in V$ is a general member.

\begin{Cor}\label{separable  of V_K}
Assume $f: X\to C$ is a fibration from a smooth  proper surface to a smooth curve over an algebraically closed field, if the generic fibre  $X_\eta/K(C)$ is geometrically integral and $V$ is a fix part free linear system on $X$, let $D\in V$ be a general member, then its horizontal part $D_h$ is reduced and separable over $C$ if the morphism $\phi: X_\eta\rightarrow \mathbb{P}(\widetilde{V}_K)$ defined by $V_K$ is separable.
\end{Cor}
\begin{proof}
Note that $D_h$ is reduced and separable over $C$ if and only if $D_\eta$ is \'etale over $K$. By Lemma~\ref{Property P}, it then suffices to prove that a general member of $V_K$ is \'etale over $K$. As $V$ is free of fix part, so is $V_K$. Therefore a general member of $V_K$ equals to $$\phi^*({\text{a general hyperplane in} \ \mathbb{P}(\widetilde{V}_K)}).$$  Now since $\phi(X_\eta)$ is geometrically integral (hence only have finitely many non-smooth points over $K$) and $\phi$ is separable (hence \'etale outside finitely many points), a general member of $V_K$ will evidently be \'etale over $K$. \end{proof}
\begin{Rem} Let $V,D$ be as above, \begin{enumerate}
\item if $p\nmid D\cdot F$ ($F$ is a fibre of $X/C$), then $\phi$ is automatically separable. 
\item if $V$ is not composed with pencils, then $D$ is furthermore irreducible by \cite{J} Theorem 6.11.
\end{enumerate}
\end{Rem}

\subsection{Some other supplementaries}\subsubsection{}\label{subsection: ramification}
Let $k$ be an algebraically closed field of characteristic $p$, and $\phi:D\rightarrow C$ be  \emph{a separable morphism} between two smooth curves over $k$. Assume  $d\in D$ is a closed point and $c:=\phi(d)$. Choose an arbitrary uniformizer $s\in \cO_{c,C}$ of $c$.
\begin{Def}\label{notation}
We define the \emph{ramification index} of $\phi$ at $d$ to be the number $R_d(\phi):=\dim_k (\Omega_{D/C})_d$. 

And we define the \emph{type of ramification} at $d$ to be a set $\Lambda_{d}(\phi)$ of numbers as below.

\begin{enumerate}

\item If $\phi$ is wildly ramified at $d$,  $\Lambda_d(\phi):=\{v(s), R_d(\phi)\}$, where $v$ is the normalised valuation at $d$. Note here that $v(s)$ is independent on the choice of $s$ and $p\mid v(s)$ by assumption, we also define $j_d(\phi):= v(s)/p$.

\item If $\phi$ is tamely ramified at $d$,  $\Lambda_d(\phi):=\{R_d(\phi)\}$. Note that in this case $p\nmid v(s)=R_d(\phi)+1$.\qed
\end{enumerate}
\end{Def}

When no confusion can occur, we shall use $R_d$ and $\Lambda_d$ instead of $R_d(\phi)$ and $\Lambda_d(\phi)$.
\begin{Rem}
By abuse of language we can also talk about the ramification index and ramification type of a certain kind of  function as below. Suppose $s\in \cO_{d,D}\backslash \cO_{d,D}^p$ is an element in the maximal ideal of $\cO_{d,D}$, then we can define a separable local morphism (still denoted by $s$) $s: \Spec(\cO_{d,D})\to \Spec(k[x])_{(x)}$ mapping $x\mapsto s$. By mixing the function $s$ and the associated morphism $s$ we are allowed to talk about its ramification index $R_d(s)$ and ramification type $\Lambda_d(s)$.       
\end{Rem}
From our definition of ramification index, we have Hurwitz's formula:
\begin{Prop}[see, \cite{LQ} Theorem 4.16 and Remark 4.17]
Suppose $\phi:D\rightarrow E$ is a separable morphism between smooth projective curves. Then 
\begin{equation}\label{Hurwitz}
2\deg\phi(g(E)-1)+\sum\limits_{d}R_d(\phi)=2g(D)-2.
\end{equation} \qed
\end{Prop}

\subsubsection{}
Finally, to close this section, we shall recall the following variation of Clifford's theorem.
\begin{Lem}[\cite{Be}, Lemme~1.3]\label{clifford}
Let $C$ be a smooth projective curve of genus $q:=g(C)$, and let $D\ge 0$ be an effective divisor, then either 
\begin{enumerate}
\item $\deg D>2(q-1)$, and $\deg D=h^0(\cO_C(D))+q-1$; or

\item $2(h^0(\cO_C(D))-1) \le \deg D \le 2(q-1)$. In particular this time we have $h^0(\cO_C(D))\le q$.
\end{enumerate}
\end{Lem}




\section{Examples}\label{section 2}
In this section we will present some examples of surfaces of general type with negative $c_2$ and calculate some of their numerical invariants.

\subsection{Examples of M. Raynaud}\label{Raynaud's example}
Let us briefly recall the examples of M. Raynaud \cite{R}.

Let $k$ be an algebraically closed field of characteristic $p>2$, and assume $C$ is a  smooth projective curve of genus $q\ge 2$ such that there is an $f\in K(C)$ satisfying $(df)=pD$ for some divisor $D$. Let $\cL=\cO_C(D)$, $l=\deg D$ and $\cM$ be any invertible sheaf on $C$ such that $\cM^{\otimes 2}\simeq \cL$. We have $m:=\deg \cM=l/2$ and $2q-2=pl=2pm$.

By \cite{R} Proposition~1, we can find a rank $2$ locally free sheaf $\cE$ and its associated ruled surface $\rho: Z:=\mathbb{P}(\cE)\rightarrow C$ such that 
\begin{enumerate}
\item $\det(\cE)\simeq \cL$, in particular $\cO(1)^2=l$;
\item there is a section $\Sigma_1\in |\cO(1)|$;
\item there is a multi-section $\Sigma_2$ such that the canonical morphism $\rho: \Sigma_2\rightarrow C$ is isomorphic to the Frobenius morphism. 
\item $\Sigma_1\cap \Sigma_2=\emptyset$,
\item $\cO_Z(\Sigma_2)=\cO(p)\otimes \rho^*(\cL^{\otimes {-p}})$.
\end{enumerate}
 
Let $\Sigma:=\Sigma_1+\Sigma_2$, then $\Sigma$ is a nonsingular divisor of $Z$, and $$\cO_Z(\Sigma)=\cO(p+1)\otimes \rho^*(\cL^{\otimes{-p}})=(\cO(\frac{p+1}{2})\otimes \rho^*(\cM^{\otimes {-p}}))^{\otimes 2},$$ hence the data $\{\cO(\frac{p+1}{2})\otimes \rho^*(\cM^{\otimes {-p}}), \Sigma\in |(\cO(\frac{p+1}{2})\otimes \rho^*(\cM^{\otimes {-p}}))^{\otimes 2}|\}$ defines a flat double cover $\pi: S\rightarrow Z$ by Construction \ref{Ls}. 

\begin{Prop} We have
\begin{enumerate}
\item $K_Z=\cO(-2)\times (\rho^*\cL^{\otimes p+1})$, and $K_S=\pi^*(\cO(\frac{p-3}{2})\otimes \rho^*\cM^{\otimes p+2});$
\item $\chi(\cO_S)=(p^2-4p-1)l/8$, $K_S^2=(3p^2-8p-3)l/2$, and $c_2(S)=-4(q-1)$;
\item $S$ is a minimal surface of general type if $p\ge 5$.
\end{enumerate}
\end{Prop}
\begin{proof}
By Proposition \ref{Property}, $S$ is regular.

(1) Since $\det\cE=\cL$, $\Omega_{C/k}\simeq \cL^{\otimes p}$, we immediately get $$K_Z=\cO(-2)\otimes \rho^*\cL^{\otimes p+1},$$ then by Corollary \ref{formula of dualizing sheaf}, $$\omega_{S/Z}=\pi^*(\cO(\frac{p+1}{2})\otimes \rho^*\cM^{\otimes -p}),$$ hence $$K_S=\pi^*(\cO(\frac{p-3}{2})\otimes \rho^*\cM^{\otimes p+2}).$$

(2) By Corollary \ref{genus}, we have $$\chi(\cO_S)=2\chi(\cO_Z)+\frac{\Sigma^2+2\Sigma\cdot K_Z}{8}=\frac{p^2-4p-1}{8}l,$$ and $$K_S^2=\pi^*(\cO(\frac{p-3}{2})\otimes \rho^*\cM^{\otimes p+2})^2=2(\cO(\frac{p-3}{2})\otimes \rho^*\cM^{\otimes p+2})^2=\frac{3p^2-8p-3}{2}l,$$ therefore $c_2(S)=12\chi(\cO_S)-K_S^2=-2pl=-4(q-1)$.

(3) If $p\ge 5$, then any closed fibre of $S\rightarrow C$ is irreducible and  has arithmetic genus $(p-1)/2$, hence $S$ is a minimal surface of general type. 
\end{proof}
\begin{Rem}\begin{enumerate}
\item Note that the fibration $S\rightarrow C$  is uniruled. In this case we do not have the positivity of the dualizing sheaf $\omega_{S/C}$ (compare with \cite{SZ} \S~2). We shall point out that $\omega_{S/C}$ here is \emph{not nef}. In fact $\omega_{S/C}=\omega_{S/Z}\otimes \pi^*\omega_{Z/C}=\pi^*(\cO(\frac{p-3}{2})\otimes \rho^*\cM^{\otimes 2-p})$, however $$\Sigma_1\cdot (\cO(\frac{p-3}{2})\otimes \rho^*\cM^{\otimes 2-p})=-l/2<0.$$

\item Note that $$\frac{\chi(\cO_X)}{K_X^2}=\frac{p^2-4p-1}{4(3p^2-8p-3)}.$$ This number is exactly our conjectural $\kappa_p$ (Conjecture \ref{Conjecture}).

\item If $p=3$, Raynaud's example is an quasi-elliptic surface, hence it is not of general type. This is one of the reasons why we can find $\kappa_5$ but not $\kappa_3$.
\end{enumerate}
\end{Rem}

\subsection{Examples in characteristic $2, 3$}\label{examples 2}
First we give an example of surfaces with negative $c_2$ over a field $k$ of characteristic $3$.
Choose $m=3^n-1$ points, say $t_1,...,t_m$ on $\mathbb{A}_k^1=\mathbb{P}_k^1\backslash\{\infty\}$, and  we can construct
a cyclic cover $C\rightarrow \mathbb{P}_k^1$ of degree $m$ such that the branch locus equals to $B:=\sum\limits_i t_i$ canonically as we did before for flat double covers (see Construction \ref{Ls}). In particular by Hurwitz's formula,  $q-1: =g(C)-1=(3^n-1)(3^n-4)/2$

Let $Y:=\mathbb{P}_C^1$,  $p_1: Y\rightarrow C, $ and $p_2: Y\rightarrow \mathbb{P}_k^1$ be the canonical projections. Let $\Pi_1$ be the divisor $C\times_k \{\infty\}$, and $\Pi_2$ be the divisor which is the image of $C\stackrel{F^n\times h }{\longrightarrow}C\times_k \mathbb{P}_k^1=Y$, here $F^n$ is the $n$-th Frobenius morphism. Then $\Pi:=\Pi_1+\Pi_2$ is an even divisor (\ie, $\Pi=2D$ for some divisor $D$), in particular we can define a flat double cover $\pi: S'\rightarrow Y$ whose branch locus equals to $\Pi$. 

\begin{Prop}
 Let $S$ be the minimal model of $S'$, when $n\ge 2$, $S$ is of general type and $c_2(S)\le -4(q-1)+3m$.
\end{Prop} 
\begin{proof}[Sketch of the proof]
We consider the canonical resolution of $S$. We have $\Pi_1$ and $\Pi_2$ intersect properly, and the singularities of $\Pi_2$ are the pre-images of $B$. Blowing up these points($2m$ points in total), we get the desingularization of $\Pi$. Consequently we get a desingularization $S_1\to S'$. It is clearly $S_1\rightarrow C$ has $2m$ non-irreducible fibres (each has $2$ components), therefore we have $$c_2(S)\le c_2(S_1)=-4(q-1)+3m$$ by Grothendieck-Ogg-Shafarevich formula (see formula (\ref{estimation of c_2}) below).
\end{proof}
\begin{Rem}
When $n\rightarrow +\infty$, we see that $c_2(S)/(q-1)\rightarrow -4$.
\end{Rem}

We mention that in characteristic $2$ there are also surfaces of general type with negative $c_2$. One example is \cite{Li1}, Theorem 7.1, where $c_1^2=14, \chi=1$ and $c_2=-2$.

\section{Surfaces of general type with negative $c_2$}\label{section: surfaces 0}
Let $k$ be any algebraically closed field of characteristic $p>0$, and let $X$ be a minimal surface of general type with negative $c_2(X)$.  We first recall a theorem of N. Shepherd-Barron on the structure of the Albanese morphism of $X$.

\begin{Thm} [\cite{SB2} Theorem 6]\label{Albanese}
The Albanese morphism of $X$ factors through a fibration $f: X\rightarrow C$ such that:

(1) $C$ is a nonsingular projective curve of genus $q:=g(C)\ge 2$, $f_{*}\cO_{X}\simeq \cO_{C}$, and $\Alb_{X}\simeq \Alb_{C}$.

(2) The geometric generic fibre of $f$ is an integral singular rational curve with unibranch singularities only. 
\end{Thm} 

We then introduce the following notation according to this theorem:
\begin{enumerate}[a)]
\item  $K:=K(C)$ (resp. $\overline{K}$; $\eta$;  
$\overline{\eta}$) is the function field of $C$ (resp. a fixed algebraic closure of $K$;
the generic point of $C$;  a fixed geometric generic point of $C$);
    
\item  $F:$ a general fibre of $f$;
    
\item  $g:=p_{a}(F)$ is the arithmetic genus of any fibre of $f$;
    
\item  $p_g:=h^2(X,\cO_X)$ is the geometric genus of $X$.

\item  $q(X):=h^{1}(X,\mathcal{O}_{X})$ is the irregularity of $X$. Since $\Alb_{X}\simeq \Alb_{C}$, we have the following inequality due to Igusa \cite{Ig},
\begin{equation} 
q(X)\ge \dim \mathrm{Alb}_{X}=\dim \mathrm{Alb}_{C}=q;
\end{equation}
\item  Denote by $Z$  the fixed part of $|K_X|$, $Z_h$  the horizontal part of $Z$ and $Z_0:=(Z_h)_{\mathrm{red}}$;
    
\item  Let $f^*(\Omega_{C/k})(\Delta)$ be the saturation of the injection $f^*\Omega_{C/k}\to \Omega_{X/k}$. Define $N:=f^*K_C+\Delta$ to be the divisor class of $f^*(\Omega_{C/k})(\Delta)$. It is well known that $\Delta$ is supported on the non-smooth locus of $f$, in particular each prime horizontal component of $\Delta$ is inseparable over $C$. 

\item For any effective divisor $D$ on $X$, we will use both $D_\eta$ and $D|_{X_\eta}$ to denote its restriction to the generic fibre of $f$ and we use $D_h,D_v$ to denote its horizontal and vertical part.
\end{enumerate}
If let $S:=X_{\eta}/K$, then by our construction we have $\cO_{\Delta_\eta}\simeq\cA:=(\Omega_{X_\eta/K})_{\mathrm{tor}}$ and $\cO_X(\Delta)|_{X_{\eta}}\simeq\det \cA$. Therefore Corollary~\ref{complete} implies the following lemma.
\begin{Lem}\label{degree of Delta}
We have \begin{enumerate}
\item $(p-1)\mid 2g$;
\item if $g<(p^2-1)/2$, then $\deg_K(\Delta_\eta)=2pg/(p-1)$;
In particular, if $g=(p-1)/2$, then  $\Delta_h$ is integral.
\end{enumerate}
\end{Lem}

From Noether's formula (\ref{Noether formula}), to bound $\kappa_p$ from below, we only have to bound  \emph{$\lambda(X):=K_X^2/(q-1)$} and  \emph{$\gamma(X):=c_2(X)/(q-1)$}. One lower bound of $\gamma(X)$ comes out naturally once we apply Grothendieck-Ogg-Shafarevich formula (\cite{G}, Expos\'e X) to  $X$ to obtain the following formula:

\begin{equation}\label{estimation of c_2}
c_{2}(X)=-4(q-1)+\sum\limits_{c\in |C|}(b_{2}(X_{c})-1)\geq -4(q-1).
\end{equation}
Here we note that $H^{1}_{\mathrm{\acute{e}t}}(X_{\overline{\eta}}, \mathbb{Q}_l)=0$, as  $X_{\overline{\eta}}$ is a rational curve with unibranch singularities only, hence the Swan conductor and $b_1(X_c)$ both vanish. By the way, this formula also shows that $X$ is supersingular in the sense of Shioda.
\begin{Prop}\label{supersingular}
The surface $X$ is supersingular in the sense that $b_2(X)=\varrho(X)$, here $\varrho(X)$ is the Picard number of $X$.
\end{Prop}
\begin{proof}
Using the fibration $f:X\rightarrow C$ we have $$\varrho(X)\ge 2+\sum\limits_{c\in |C|} (^\sharp\{\text{irreducible components of} \ X_c\}-1)=2+\sum\limits_{c\in |C|} (b_2(X_c)-1).$$ Conversely since $b_1(X)=2q$ and $c_2(X)=2-2b_1+b_2$, we get $$b_2=2+\sum\limits_{c\in |C|} (b_2(X_c)-1)\le \varrho(X)$$ from the (\ref{estimation of c_2}). Hence $b_2=\varrho(X)$ and $X$ is supersingular.  
\end{proof}
\begin{Rem}
Since $X$ is dominated by a ruled surface, Proposition~\ref{supersingular} can also be derived from Lemma of \cite{Sh1} \S2.
\end{Rem}

It remains to find  lower bounds of $\lambda(X)=K_X^2/(q-1)$. Note that pulling back by an \'etale cover of $C$, $\lambda(X)$ is invariant while $q-1$ and $K_X^2$ are multiplied by the degree of the cover, thus we can assume $$q\gg \lambda(X)>0, \ K_X^2\gg 0.$$ We shall first go through N. Shepherd-Barron's method in \cite{SB2} quickly, based on which we will give an improvement.
\begin{Lem}\label{generalisation of SB}
Assume $H$ is a reduced horizontal divisor on $X$ such that any of its irreducible component is separable over $C$, then we have:
$$N\cdot H\le (H+K_X)\cdot H.$$ 
\end{Lem}
\begin{proof}
We consider the morphism 
$\cO_X(N)|_{H}\rightarrow \omega_{H/k}$ given by the composition $\cO_X(N)|_{H}=f^*(\Omega_{C/k})(\Delta)|_{H}\rightarrow \Omega_{X/k}|_{H} \rightarrow \Omega_{H/k} \rightarrow \omega_{H/k}$. We show that under our assumption this morphism is injective. 
Taking the degrees in $\cO_X(N)|_{H}\hookrightarrow \omega_{H/k}$ will then imply that $N\cdot H\leq \deg(\omega_{H/k})=(K_{X}+H)\cdot H$.

Let $\xi_i\in X_{\eta}$ be the generic point of an irreducible 
component $H_{i}$ of $H$.  Then $\xi_{i}$ belongs to the smooth locus of $X_{\eta}/K$, 
so $(\cO_X(N)|_H)_{\xi_i}\to \omega_{H/k,\xi_i}$ coincides with 
$(f^*\Omega_{C/k})_{\xi_i}\to \Omega_{H/k, \xi_i}$ and the latter is injective because $H_i\to C$ is separable. So the kernel of $\cO_X(N)|_H\to \omega_{H/k}$ is
a skyscraper sheaf. As $\cO_X(N)$ is an invertible sheaf and $H$ has no embedded
points (it is locally complete intersection), the kernel is trivial and 
$\cO_X(N)|_H\to \omega_{H/k}$ is injective. 
\end{proof}

\begin{Cor}\label{Corollary after basic estimation}
If the complete linear system $|H|$ is free of fixed part and defines a separable generically finite map, then $N\cdot H\le (H+K_X)\cdot H$. 
\end{Cor}
\begin{proof}
It suffices  to show that a general member of $|H|$ is integral and separable over $C$, but this follows immediately from Corollary~\ref{separable  of V_K}.
\end{proof}
With the help of \cite{SB} Theorems 24, 25, 27 and under our assumption $q\gg \lambda(X)>0, K_X^2\gg 0$,  we then see that the linear systems 
\begin{enumerate}
\item $|2K_{X}|$, for $p>2, g> 2$;
\item $|3K_{X}|$, for $p=2, g>2$;
\end{enumerate}
are base point free and define birational morphisms. 
Applying Lemma~{\ref{generalisation of SB} to the above linear systems, we then obtain:
\begin{Cor}\begin{enumerate}
\item If $p\ge 3, g\ge 3$, then 
\begin{equation}\label{basic estimation}
4(g-1)(q-1)+K_X\cdot \Delta_h\le 3K_X^2.
\end{equation}
\item If $p=2, g\ge 3$, then 
\begin{equation}\label{basic estimation 2}
4(g-1)(q-1)+K_X\cdot \Delta_h\le 4K_X^2.
\end{equation}
\end{enumerate}
\end{Cor}

From these inequalities, we immediately get that
\begin{Cor}[N. Shepherd-Barron]\label{naive estimation}
\begin{enumerate}
\item If $p\ge 3, g\ge 3$, then $K_X^2> 4(g-1)(q-1)/3$;
\item If $p=2, g\ge 3$, then $K_X^2> (q-1)(q-1)$.
\end{enumerate}
\end{Cor}

We now begin to improve this estimation of $\lambda=K_X^2/(q-1)$ by considering its canonical system $|K_X|$.
\begin{Lem}\label{lower bound of p_g}
We have $p_g> 2(q-1)/3$. 
\end{Lem}
\begin{proof} We have\begin{equation}\label{eq: p_g}\begin{split}
p_g-1&=\chi(\cO_X)-1+(q(X)-1)\\&\ge \frac{K_X^2-4(q-1)}{12}-1+(q-1)\\&=\frac{K_X^2+8(q-1)-12}{12},
\end{split}
\end{equation} hence $p_g> 2(q-1)/3$.
\end{proof}

\begin{Lem}\label{pencil equals to C}
If  $|K_X|$ is composed with a pencil, then $|K_X|=Z+f^*|M|$, where $M$ is a divisor on $C$ such that $h^0(C,M)=p_g$, and $$K_X^2\ge \min\{4(p_g-1)(g-1), 2(p_g+q-1)(g-1)\}.$$
\end{Lem}
\begin{proof}
Assume $|K_X|$ is composed with a pencil. If the pencil is not $C$, then $K_X\sim_{alg} Z+aV$, with $a\ge p_g-1$ and $V$ is an integral divisor dominating $C$.  
So by \cite{E} Proposition~1.3, we have either  $$K_X^2\ge 2a(p_a(V)-1)\ge 2(p_g-1)(q-1)>\lambda(X)(q-1)=K_X^2,$$ or  $$K_X^2\ge a^2\ge (2(q-1)/3-1)^2>\lambda(X)(q-1)=K_X^2,$$ a contradiction. Here we have used our assumption $q-1\gg \lambda(X)$ and Lemma~\ref{lower bound of p_g}. So the pencil is $C$, therefore $|K_X|=Z+f^*|M|$ and $h^0(C,M)=p_g$. The inequality $$K_X^2\ge K_X\cdot f^*M=(2g-2)\deg M \ge \min\{4(p_g-1)(g-1), 2(p_g+q-1)(g-1)\}$$ follows from Lemma~\ref{clifford}.
\end{proof}
\begin{Thm}\label{p>7 theorem}
If $p\ge 7$, then there is a positive number $\epsilon$ (depending on $p$ only) such that $K_X^2\ge (p-3+\epsilon)(q-1)$.
\end{Thm}
\begin{proof}
Since $(p-1)\mid 2g$, we have either $g\ge (p-1)$ or $g= (p-1)/2$. When  $g\ge p-1$,  it follows from Corollary~\ref{naive estimation} that $K_X^2>4(g-1)(q-1)/3\ge 4(p-2)(q-1)/3$.

Assume $g=(p-1)/2$. If $|K_X|$ is composed with a pencil, then  $K_X^2\ge \min\{2(p_g-1)(p-3), (p_g+q-1)(p-3)\}> (p-3+\epsilon)(q-1)$ for some $\epsilon>0$ by Lemma~\ref{lower bound of p_g} and \ref{pencil equals to C}. Now we assume $|K_X|$ is not composed with pencils. Choose a general member $D'\in |K_X-Z|$. Since $D'\cdot F\le K_X\cdot F=2g-2=p-3<p$,  $D'$ is integral and separable over $C$ by Lemma~\ref{separable  of V_K} and its remark.   
Note that $Z_0$ is also separable over $C$, we can apply Lemma~\ref{generalisation of SB} to $H=D'+Z_0$ to obtain
$$(K_X+H)\cdot H\ge H\cdot N.$$
Let $Z_h=\sum\limits_i r_i E_i$, and $G=\sum\limits_i(r_i-1)E_i$, then $Z_0=Z_h-G=\sum\limits_i E_i$, so we have 
\begin{align*}
 H\cdot N =&H\cdot f^*K_C+H\cdot \Delta\\=& 2(p-3)(q-1)-2\sum\limits_i(r_i-1)(q-1)\deg_K (E_i)_\eta+H\cdot \Delta\\
\ge& 2(p-3)(q-1)-2\sum\limits_i(r_i-1)(q-1)\deg_K (E_i)_\eta+H\cdot \Delta_h.
\end{align*}
On the other hand \begin{align*}
(K_X+H)\cdot H &=2K_X^2-2K_X\cdot (G+Z_v)-H\cdot (G+Z_v)\\
&\le 2K_X^2-2K_X\cdot G-\sum\limits_i(r_i-1)E_i^2\\
&=2K_X^2-\sum\limits_i 2(r_i-1)(p_a(E_i)-1)-K_X\cdot G\\
&\le 2K_X^2-2\sum\limits_i (r_i-1)(q-1)\deg_K (E_i)_\eta-K_X\cdot G.
\end{align*}
Here  we note that since $E_i$ is separable over $C$,  $2p_a(E_i)-2\ge 2\deg_K(E_i)_\eta(q-1)$. 
Combining the two inequalities we get
\begin{equation}\label{H,G}
K_X^2\ge (p-3)(q-1)+H\cdot \Delta_h/2+K_X\cdot G/2.\end{equation}

If $G\neq 0$, then $K_X\cdot G/(q-1)$ will be bounded from below by a positive number depending only on $p$ (see Lemma~\ref{estimation on intersection number} below), so by (\ref{H,G}) $K_X^2/(q-1)-p+3$ will be bounded from below by a positive number $\epsilon$ depending on $p$. 

Now we only have to deal with the case where $G=0$.  By construction, we have  $F\cdot ((p-3)\Delta_h-pH)=0$, hence by Hodge Index Theorem we have $$((p-3)\Delta_h-pH)^2\le 0,$$ or $$ (p-3)\Delta_h^2/2p+pH^2/2(p-3)\le\Delta_h\cdot H .$$ Note that this time $H$ is a horizontal part of an element in $|K_X|$, hence $K_X^2\ge H^2$,  so from \begin{align*}
(3p-12)K_X^2/2(p-3) + pH^2/2(p-3)&\ge K_X^2+H^2\ge (K_X+H)\cdot H \\&\ge N\cdot H\ge 2(p-3)(q-1)+H\cdot\Delta_h
\end{align*} we see that $$(3p-12)K_X^2/2(p-3)\ge 2(p-3)(q-1)+(p-3)\Delta_h^2/2p.$$ Combining with (\ref{basic estimation}) and the fact $K_X\cdot \Delta_h+\Delta_h^2=2p_a(\Delta_h)-2\ge 2(q-1)$, we have 
$$(\frac{3(p-3)}{2p}+\frac{3p-12}{2(p-3)})K_X^2\ge (\frac{p-3}{p}+2(p-3)+\frac{(p-3)^2}{p})(q-1),$$or$$K_X^2\ge \frac{6p^2-22p+12}{6p^2-30p+27}(p-3)(q-1)>(p-3+\epsilon)(q-1).$$
\end{proof}
\begin{Lem}\label{estimation on intersection number}
Let $B$ be an horizontal prime divisor with $r:=[K(B)\cap K^{\mathrm{sep}}:K]$, then $K_X\cdot B+B^2\geq 2r(q-1)$.  In particular $$K_X\cdot B\ge \sqrt{2r(q-1)K_X^2+(K_{X}^2)^2/4}-K_{X}^2/2=(\sqrt{\lambda^2+8r\lambda}-\lambda)(q-1)/2,$$ here $\lambda=\lambda(X)$.   
\end{Lem}
\begin{proof}
It is well known that $2(p_a(B)-1)\geq 2(p_a(B')-1)\geq 2r(p_a(C)-1)$, where $B'$ is the normalisation of $B$ (see \cite{LQ}, pp 289-291). So $K_X\cdot B+B^2=2(p_a(B)-1)\geq 2r(q-1)$. 

(i) If $B^2\leq 0$, clearly $K_X\cdot B\geq 2r(q-1)>\sqrt{2r(q-1)K_X^2+(K_{X}^2)^2/4}-K_{X}^2/2$; 

(ii) If $B^2> 0$, $B$ is nef and $(K_X\cdot B)^2\ge B^2K_X^2$, hence $$(K_X\cdot B)^2/K_X^2+(K_X\cdot B)^2\ge 2r(q-1),$$ so $K_X\cdot B> \sqrt{2r(q-1)K_X^2+(K_{X}^2)^2/4}-K_{X}^2/2$. 
\end{proof}
\begin{Cor}\label{kappa p>7}
If $p\ge 7$, then $\kappa_p> (p-7)/12(p-3)$.
\end{Cor}

Next we apply this method to the cases $p=3,5$.

\subsection{Case $p=5$}
The case $p=5$ is very special in that we can indeed find out $\kappa_5=1/32$. The main reason is that the smallest possible value of $g$ is $(p-1)/2=2$, in which case $X_\eta$ will  automatically be hyperelliptic. We  carry out a calculation of $\chi(\cO_X)$ in the hyperelliptic case in the next section, which provides a more precise lower bound of $\chi(\cO_X)/(q-1)$, and consequently gives the precise value of $\kappa_p$ when $g=2$. In this subsection we aim to deal with the cases $g>2$ and show that  $\chi/c_1^2\ge 1/32$ also holds in these situations, this result combining with the result in the next section (Theorem~\ref{evidence}) will then imply $\kappa_5=1/32$ (Corollary~\ref{last}).   

Notice that following  Noether's formula and (\ref{estimation of c_2}), in order to prove $\chi/c_1^2\ge 1/32$ it suffices to show $K_X^2\ge 32(q-1)/5$.  When $g\ge 6$, this inequality follows immediately from Corollary~\ref{basic estimation}. So we are left to deal with the case $g=4$. So we assume $g=4$ in the sequel of this subsection.

Let $$i: H^{0}(X,K_{X})\hookrightarrow H^{0}(X_{\eta},K_{X}|_{X_{\eta}})\simeq H^0(X_\eta,\omega_{X\eta/K})$$ be the canonical restriction map,  $V:=|K_X|$, and $V_K$ be its restriction \ie \  $V_K\subset |\omega_{X_\eta/K}|$ is the sub-linear system associated to the $K$-subspace spanned by $\mathrm{Im}(i)$ (see Subsection~\ref{Bertini}).

\begin{Lem}\label{local 5,4}
If $|K_X|$ is not composed with a pencil, and $D'\in |K_X-Z|$ is a general member, then either 
\begin{enumerate}
\item $D'$ is integral and separable over $C$; or 
\item $Z_h$ is a section of $f$, and $D'^2\ge 5(p_g-2)$.
\end{enumerate}
\end{Lem}
\begin{proof}
We consider the morphism $\phi: X_\eta\rightarrow \mathrm{P}^{r-1}_K$ defined by $V_K$, here $r$ is the dimension of the $K$-subspace spanned by $\mathrm{Im}(i)$ (note that $r=1$ will imply $|K_X|$ is composed with pencils). Note that by construction, we have a formula $$\deg\phi\deg(\phi(X_\eta))+\deg_K Z_\eta=\deg_K\omega_{X_\eta/K}=6.$$

(1) If $\phi$ is separable, then $D'$ is integral and separable over $C$ by Lemma~\ref{separable  of V_K} and its remark.

(2) If $\phi$ is not separable, then $\deg \phi=5$, $Z_\eta$ is  therefore a rational point, hence $Z_h$ must be a section. On the other side since we have $\deg(\phi)|\deg(\phi_{|K_X-Z|})$, here $\phi_{|K_X-Z|}$ is the canonical map of $X$, then  $\deg(\phi_{|K_X-Z|})\ge 5$ and hence $D'^2\ge 5(p_g-2)$ by \cite{E}, Proposition~1.3(ii). 
\end{proof}

\begin{Thm}\label{5,4 K^2>6.4}
Under the hypothesis $g=4$, $K_X^2\ge 32(q-1)/5$.
\end{Thm}
\begin{proof}
(1) If $|K_X|$ is composed with a pencil, then $|K_X|=Z+f^*|M|$, and $\deg M=p_g+q-1>2(q-1)$ by Lemma~\ref{pencil equals to C} (Note that $\chi(\cO_X)> 1$ and hence $p_g>q$ by Lemma~\ref{basic estimation} and assumption $q\gg 0$). So we have $$K_X^2\ge K_X\cdot f^*M=6\deg M=6(p_g+q-1)>12(q-1).$$ 

(2) Suppose $|K_X|$ is not composed with a pencil and a general member $D'\in |K_X-Z|$ is integral and separable over $C$. Then $D'+Z_0$ is the sum of reduced divisors separable over $C$.  We can apply Lemma \ref{generalisation of SB} to the divisor $H:=D'+Z_0$. Assume $Z_h=\sum\limits_i r_iZ_i$, and let $G:=Z_h-Z_0=\sum\limits_i (r_i-1)Z_i$, then in the similar way for inequality (\ref{H,G}), we can obtain
$$2K_X^2\ge 12(q-1)+\sum\limits_i(r_i-1)K_X\cdot Z_i+H\cdot \Delta_h.$$ Note that $H\cdot \Delta_h\ge 0$ as no component of $Z_0$ could be inseparable over $C$. In particular $K_X^2/(q-1)\ge 6$  and consequently $$K_X\cdot Z_i>(\sqrt{21}-3)(q-1)>3(q-1)/2$$ by Lemma~\ref{estimation on intersection number}. So if $K_X^2\le 32(q-1)/5$, we must have $r_i=1$ for all $i$, namely $G=0$. Then a similar trick as we did to deal with the case $G=0$ in the proof of Theorem~\ref{p>7 theorem} will implies $K_X^2> 32(q-1)/5$, contradiction.

(3) Suppose $|K_X|$ is not composed with a pencil, $Z_h$ is a section and $D'^2\ge 5(p_g-2)$. Then \begin{equation}\label{local 5-4-1}
K_X^2\ge D'^2+K_X\cdot Z_h\ge 5(p_g-2)+K_X\cdot Z_h.
\end{equation}  

 In particular, $$K_X^2\ge 5(p_g-2)\ge 5(K_X^2+8(q-1)-24)/12,$$ hence $$K_X^2\ge (40(q-1)-120)/7\ge 39(q-1)/7,$$ as $q\gg  0$ by assumption. Combining this with Lemma~\ref{estimation on intersection number} we obtain $$K_X\cdot Z_h\ge (\sqrt{3705}-39)(q-1)/14\ge 3(q-1)/2.$$ Returning back to  (\ref{local 5-4-1}) and using (\ref{eq: p_g}) again,
we have $$K_X^2\ge 5(K_X^2+8(q-1)-24)/12+3(q-1)/2,$$ which implies
 $K_X^2>32(q-1)/5$ as $q\gg 0$ by assumption.
 
Lemma~\ref{local 5,4} shows that the three cases  above are exhaustive.
\end{proof}
\begin{Cor}\label{main 5,4}
If $g\ge 4$, then $\chi/c_1^2\ge 1/32$.
\end{Cor}

\subsection{Case $p=3$}
As another application of our method, we show $\kappa_3>0$ in this subsection. It suffices to prove that there is some positive number $\epsilon_0$ independent on $X$ such that $K_X^2\ge (4+\epsilon_0)(q-1)$ holds. Following Corollary~\ref{naive estimation}, this inequality holds automatically if $g\ge 4$. So we divide our discussions into cases $g=2$ and $3$.

\subsubsection{\textbf{Case $g=3$}}
\begin{Lem}\label{canonical map,3,3}
One of the following properties is true:
\begin{enumerate}
\item $|K_X|$ is composed with a pencil.

\item $|K_X|$ is not composed with a pencil, $Z_h$ is reduced and a general member $D'\in |K_X-Z|$ is integral and separable over $C$;

\item $Z_h$ is a section and $(K_X-Z)^2\ge 3(p_g-2)$.
\end{enumerate}
\end{Lem}
\begin{proof}
Assume $|K_X|$ is not composed with a pencil. Let $V=|K_X|$ and $V_K$ be its restriction to the generic fibre. Then $B:=Z_\eta$ is the fixed part of $V_{K}$. Let $\phi: X_{\eta}\rightarrow \mathbb{P}^{r-1}_K$. Note that as in case $p=5$, we have a formula $$\deg\phi\deg(\phi(X_\eta))+\deg_KB=\deg_K\omega_{X_\eta/K}=4.$$ Hence if $\deg_K B\ge 2$, we must have either $\deg\phi=2, \deg(\phi(X_\eta))=1$ or $\deg\phi=1, \deg(\phi(X_\eta))=2$. This first case implies that $X_\eta$ is quasi-elliptic,  contradiction to Lemma~\ref{on Pa}, the second implies $\phi(X_\eta)$ is a plane conic, which is indeed smooth since it is geometrically integral, and $X_\eta$ is birational to this plane conic,  contradiction. So $\deg_K B\le 1$, hence $Z_h$ is reduced.   

If $\phi$ is separable, then a general member  $D'\in|K_X-Z|$ is as stated in part (2) of our lemma by Lemma~\ref{separable  of V_K}.

If $\phi$ is inseparable, then $\deg\phi=3$, $\deg B=1$. So $Z_h$ is a section. Note that in this case the canonical map $\phi_{|K_X|}=\phi_{|K_X-Z|}$ of $X$ is also inseparable, hence its degree is at least $3$, therefore $(K_X-Z)^2\ge 3(p_g-2)$ by \cite{E} Proposition~1.3.
\end{proof}
\begin{Thm}
There is some positive constant $\epsilon_0$ independent on $X$ such that $K_X^2>(4+\epsilon_0) (q-1)$.
\end{Thm}
\begin{proof}
There are only three possibilities as below by the previous lemma.

(1). The canonical system $|K_X|$ is composed with a pencil. Then it follows from Lemma~\ref{pencil equals to C} $$K_X^2\ge 4 \min\{2p_g-2, p_g+q-1\}.$$  Combing this inequality with (\ref{eq: p_g}), we have either 
\begin{enumerate}[A)]
\item $K_X^2\ge 2(K_X^2+8(q-1)-12)/3$; or 
\item $K_X^2\ge (K_X^2+20(q-1))3$.
\end{enumerate}
Both conditions imply that $K_X^2\ge (4+\epsilon_0)(q-1)$ for some constant $\epsilon_0>0$ independent on $X$ as $q\gg 0$.

(2). The canonical system $|K_X|$ is not composed with a pencil, $Z_h$ is reduced and a general member $D'\in|K_{X}-Z|$ is integral and separable over $C$. So $D=D'+Z\in |K_X|$ and $D'+Z_h=D_h$. We can then apply Lemma~\ref{generalisation of SB} to $H=D_h$, hence $$2K_X^2\ge (K_X+D_h)\cdot D_h\ge N\cdot D_h\ge 8(q-1)+D_h\cdot \Pi,$$ here $\Pi$ is any prime component of $\Delta_h$. A similar trick as we did to deal with the case $G=0$ in the proof of Theorem~\ref{p>7 theorem} now gives $K_X^2\ge (4+\epsilon_0)(q-1)$ for some constant $\epsilon_0>0$ independent on $X$. 

(3). The canonical system $|K_X|$ is not composed with a pencil, $Z_h$ is a  section and $(K_X-Z)^2\ge 3(p_g-2)$. Then we have $$K_X^2\ge (K_X-Z)^2+K_X\cdot Z_h\ge 3(p_g-2)+K_X\cdot Z_h.$$ Note that (\ref{basic estimation}) implies $K_X^2\ge 8(q-1)/3$ and hence Lemma~\ref{estimation on intersection number} implies $K_X\cdot Z_h> 4(q-1)/3$, so we get $$K_X^2\ge 3(p_g-2)+4(q-1)/3.$$ After combining  with (\ref{eq: p_g}) and an easy computation, this inequality will soon imply $K_X^2\ge (4+\epsilon_0)(q-1)$ for some constant $\epsilon_0>0$ independent on $X$. 
\end{proof}

\subsubsection{\textbf{Case $g=2$}}
\begin{Lem}
If $g=2$, then $\Delta_h$ is reduced and $\deg_K(\Delta_\eta)=6$.
\end{Lem}  
\begin{proof}
The canonical morphism of $X_\eta/K$ here is automatically a flat double cover of $\mathbb{P}(H^0(X_{\eta}, \omega_{X_{\eta}/K}))$. Let $B\subseteq \mathbb{P}(H^0(X_{\eta}, \omega_{X_{\eta}/K}))$ be the branch divisor associated to this double cover, then $\deg B=6$ by Corollary~\ref{genus}. Note that $X_{\eta}/K$ is geometrically rational, so $\deg_{\widehat{K}} (B_{\widehat{K}})_{\mathrm{red}}=2$. Hence $B$ is either an inseparable point of degree $6$, or the sum of two inseparable points of degree $3$. Now since $\Delta_\eta$ dominates $B$ and has the same degree over $K$, $\Delta_\eta$ must be reduced.
\end{proof}

\begin{Lem}
The bi-canonical system $|2K_X|$ is base point free and a general member of $|2K_X|$ is integral and separable over $C$.
\end{Lem}
\begin{proof}
First by \cite{SB}, Theorem 25 and our assumption $K_X^2\gg 0$, we see that  $|2K_X|$ is free of base points. Everything then follows from Lemma~\ref{separable  of V_K} and its remark 
\end{proof}

From this lemma, we shall apply  Lemma \ref{generalisation of SB} to $H=2K_X$, hence
\begin{equation}\label{basic estimation in 3,2}
3K_X^2\ge 4(q-1)+K_X\cdot \Delta_h.
\end{equation}

\begin{Lem}\label{98}
Either \begin{enumerate}
\item $|K_X|$ is composed with a pencil; or

\item  $|K_X|$ is not composed with a pencil, $Z$ is vertical, and  a general member $D\in|K_X-Z|$ is an integral horizontal divisor such that $D^2\ge 2(p_g-2)$.
\end{enumerate}
\end{Lem}
\begin{proof}
Suppose $|K_X|$ is not composed with a pencil. Let  $V:=|K_X-Z|$, since $V$ has horizontal part so $1< F\cdot (K_X-Z)\le F\cdot K_X=2$, hence $Z$ is vertical. It then follows  from Lemma~\ref{separable  of V_K} and its remark that a general member $D\in V$ is integral and separable over $C$. Finally \cite{E} Proposition~1.3 show that $D^2\ge 2(p_g-2)$ as the canonical map has degree at least $2$ in this case.
\end{proof}

\begin{Thm}
We have $K_X^2>(4+\epsilon_0)(q-1)$ for some positive constant $\epsilon_0$ independent on $X$.
\end{Thm}
\begin{proof}
(1) If $|K_X|$ is composed with a pencil, then $|K_X|=Z+f^*|M|$, and $\deg M\ge \min\{2p_g-2,p_g+q-1\}$ (Lemma~\ref{pencil equals to C}).  Note in this case that the components  $\Delta_h$ is different from any component of $Z$ for sake of degree over $C$, so $$K_X\cdot \Delta\ge K_X\cdot \Delta_h=Z\cdot \Delta_h+6\deg M\ge 6\deg M.$$ Hence (\ref{basic estimation in 3,2}) shows that $$3K_X^2\ge 4(q-1)+6\deg M.$$ After combining this with (\ref{eq: p_g}) and an easy computation we obtain $$K_X^2\ge (4+\epsilon_0)(q-1).$$

(2) Suppose $|K_X|$ is not composed with a pencil. Let $D\in |K_X-Z|$ be a general member.  By Lemma \ref{generalisation of SB}, we have 
\begin{equation}\label{99}
(K_X+D)\cdot D\ge N\cdot D\ge 4(q-1)+D\cdot \Delta.
\end{equation}
Since by construction $(3D-\Delta_h)\cdot F=0$, we have $(3D-\Delta_h)^2\le 0$, {\it i.e.}
$$D\cdot \Delta_h\ge 3D^2/2+\Delta_h^2/6.$$
Combining this with(\ref{99}) and Lemma~\ref{98}
we see that \begin{align*}
K_X^2&\ge(K_X+D)\cdot D-D^2\ge 4(q-1)+D\cdot \Delta_h-D^2\\&\ge 4(q-1)+ D^2/2+\Delta_h^2/6\ge 4(q-1)+p_g-2+\Delta_h^2/6.
\end{align*}
Combining with (\ref{basic estimation in 3,2}) and (\ref{eq: p_g}), we obtain\begin{align*}
3K_X^2/2&= (3K_X^2)/6+K_X^2
\\&\ge(4+4/6)(q-1)+p_g-2+(\Delta_h^2+K_X\cdot \Delta_h)/6\\&\ge 16(q-1)/3+p_g-2\\&\ge 16(q-1)/3+(K_X^2+8(q-1))/12-2.
\end{align*}
Hence $K_X^2\ge 72(q-1)/17-24/17\ge (4+\epsilon_0)(q-1)$ as $q\gg 0$.
\end{proof}

\begin{Cor}\label{kappa_3}
We have $\kappa_3>0$.
\end{Cor}

To close this section, we mention that if combine all the theorems proven in this section, we get a proof of Theorem~\ref{Main 2} except for the last statement $\kappa_5=1/32$.

\section{Case of hyperelliptic Fibration}\label{final}
We keep the  notations of Section~4 a)-h). In this section, we calculate $\chi(\cO_X)$ directly under the assumption $p\ge 5, g=(p-1)/2$ and $X_\eta$ is quasi-hyperelliptic.  Our calculation will show that our conjectural $\kappa_p$ is indeed the best bound of $\chi/c_1^2$ for these surfaces. It is  natural to believe that those surfaces whose $\chi/(c_1^2)$ approaches $\kappa_p$ should appear in the case $g=(p-1)/2$, the smallest possible value of $g$, so somehow we have proven our conjecture for the "hyperelliptic part".

From now on we assume $X_\eta$ is quasi-hyperelliptic and $g=(p-1)/2$.

By our assumption  $X_\eta$ is a flat double cover of a smooth plane conic $P$.  Let $B\subset P$ be the branch divisor of this flat double cover, then $\deg B=p+1$ by Corollary \ref{genus}. Since $X_{\eta}/K$ is normal but not geometrically normal by assumption, $B/K$ is reduced but not geometrically reduced (Proposition \ref{Property}). Therefore $B$ contains at least one inseparable point. Consequently $B$ is the sum of a rational point and an inseparable point of degree $p$, in particular $P\simeq \mathbb{P}_K^1$.

We then identify $P$ with the generic fibre of $p_1:Z=\mathbb{P}_C^1\to C$ in a way such that the rational point contained in $B$ is the infinity point.
Here we denote by $U,V$ the two homogeneous coordinates of $\mathbb{P}^1$, and $\infty$ is defined by $V=0$.   Denote by $\varTheta_K$ the inseparable point contained in $B$, so $\Theta_K$ is defined by $U^p-hV^p$ for a certain element $h\in K\backslash K^p$.  

Let $X_{0}$ be the normalisation of $Z$ in $K(X)$, and let $\Pi$ be the branch divisor associated to this flat double cover $X_{0}\rightarrow Z$, then $B=\Pi|_{\mathbb{P}^{1}_{K}}$. Define $\Pi_{1}$ (resp. $\Pi_{2}$) to be the closure of $\infty \in B$ (resp. $\varTheta_K \in B$) in $Z$ and $\Pi_{3}$ to be the remaining vertical branch divisors.

Here by abuse of language we denote by $h$ not only  the element of $K$ mentioned above to define $\varTheta_K$ but also the unique morphism $h: C\rightarrow \mathbb{P}^{1}_{k}$ that maps $u=U/V$ to $h$ in function fields.  Define $\alpha:=\deg(h)$ and $A:=h^{*}(\infty)$, it is clear that $\deg A=\alpha$.

With some local computations we immediately obtain the next proposition on the configuration of $\Pi$.
\begin{Prop}\label{basic proposition} We have
\begin{enumerate}

\item $\Pi_{1}=C\times_{k}\infty$, and $\Pi_{1}\cap\Pi_{2}$ equals to $A\times \infty$.

\item $\mathcal{O}_{Z}(\Pi_{2})=\mathcal{O}(p)\otimes\mathcal{O}_{Z}(p_{1}^{*}A)$,  the canonical morphism $\Pi_{2}\rightarrow C$ is a homeomorphism, and the singularities of $\Pi_{2}$ are exactly the pre-image of points on $C$ where the morphism $h$ is ramified.

\item $\mathcal{O}_{Z}(\Pi_{1})=\mathcal{O}(1)$, $\Pi_{3}=p_{1}^{*}D$, for a reduced divisor $D$. Let $d:=\deg D$, then $\alpha+d$ is even, and $\mathcal{O}_{Z}(\Pi)=\mathcal{O}(p+1)\otimes p_{1}^{*}\mathcal{O}_{C}(A+D)$.

\item $\chi (\mathcal{O}_{X_{0}})=(p-3)(q-1)/2+(p-1)(\alpha+d)/4$.
\end{enumerate}
\end{Prop}
Here we note that the last statement comes from Corollary~\ref{genus-c}.

We are going to run the canonical resolution of singularities (Definition \ref{canonical}) to $X_0\rightarrow Z$ to obtain $\chi(\cO_X)$. We first need to analyze the singularities of $\Pi$.  From the above proposition, non-negligible singularities of $\Pi$ are all lying on $\Pi_2$. Since $\Pi_{2}$ is homeomorphic to $C$ via $p_{1}$, we shall  use following conventions:  if $b_{2}\in \Pi_2$ is a singularity of $\Pi$, we divide it into one of the $4$ types below according to its image $b:=p_1(b_2)\in C$, and use the notation $\xi_{b}$ to denote $\xi_{b_{2}}$ (see Definition \ref{canonical} and Definition \ref{num}, here the flat double cover is taken to be $X_0\to Z$). The $4$ types of singularities are:
\begin{enumerate}
\item[Type I] : $b\notin(A\cup D)$ and $b$ is a ramification of $h$.
The local function of $\Pi$ near $b_{2}$ is $u^p-h$ in $\mathcal{O}_{b,C}[u]$.  

\item[Type II]: $b\in D\backslash A$.
The local function of $\Pi$ near $b_{2}$ is $t(u^p-h)$ in $\mathcal{O}_{b,C}[u]$, where $t$ is a uniformizer of $\mathcal{O}_{b,C}$.

\item[Type III]: $b\in A\backslash D$.
The local function of $\Pi$ near $b_{2}$ is $v(v^p-1/h)$ in $\mathcal{O}_{b,C}[v]$

\item[Type IV]: $b\in(A\cap D)$.
The local function of $\Pi$ near $b_{2}$ is $tv(v^p-1/h)$ in $\mathcal{O}_{b,C}[v]$, where $t$ is a uniformizer of $\mathcal{O}_{b,C}$.
\end{enumerate}

Denote
 
$\mathcal{S}:=\{b \ |\  b$ is of type I, II, III or IV$\}$; 

$\mathcal{T}:=\{b \ |\  b$ is of type III or IV, and $h$ is unramified or tamely ramified at $b\}$; 

$\mathcal{W}:=\{b \ |\  b$ is of type III or IV, and $h$ is wildly ramified at $b\}$.

By (\ref{difference}),
$$
\chi(\mathcal{O}_{X})=\chi(\mathcal{O}_{X_{0}})-\sum \limits_{b\in\mathcal{S}} \xi_{b}=\frac{(p-3)(q-1)}{2}+\frac{(p-1)(\alpha+d)}{4}-\sum \limits_{b\in\mathcal{S}} \xi_{b}.
$$

Set 
 $$d_{b}:=\left\{\begin{array}{cc}
 1, & b\in D; \\ 
 0, & b\notin D. \end{array}
 \right.$$ 

Then \begin{equation}\label{fine}
\chi(\cO_X)=\frac{(p-3)(q-1)}{2}+\frac{(p-1)\alpha}{4}+\sum\limits_{b\in \mathcal{S}} (\frac{(p-1)d_b}{4}-\xi_b).
\end{equation}


Next we study in detail these four kinds of singularities.  We will find a relation between $\frac{(p-1)d_b}{4}-\xi_b$ and $R_b(h)$ for all $b$. 
In order to do this, we give a definition as follows. 
\begin{Def}
Suppose $b\in C$ is a closed point, $t\in \cO_{b,C}$ is a uniformizer, $v$ is the canonical discrete valuation and $e\in t\cO_{b,C}\backslash \cO_{b,C}^p$, we consider an arbitrary flat double cover  $S_0\to Y_0:=\Spec(\cO_{b,C}[x])$ with branch divisor $B_0=\divi(x^p-e)$ (resp. $\divi(t(x^p-e))$, $\divi(x(x^p-e))$, $\divi(tx(x^p-e))$). Let $Q$ denote the point $(x,t)$ of $\Spec(\cO_{b,C})[x]$, then we define the number $\xi_{I, e}$ (resp. $\xi_{II, e}$, $\xi_{III, e}$, $\xi_{IV, e}$) to be $\xi_Q$ with respect to this flat double (see Definition~\ref{num}).  
\end{Def}

Note that by definition we have that  $\xi_b=\xi_{*,e}$ for some $e$ such that $\Lambda_b(e)=\Lambda_{b}(h)$ (see Definition \ref{notation}), here $*$ is the type of $b$ (\ie \ I, II, III or IV).

Note also that if $R_b(e)\ge p$ (see Definition \ref{notation}), then $e=t^p(\lambda_1+e_1)$ for a unique $\lambda_1\in k$ and $e_1\in t\cO_{b,C}$. In particular $R_b(e_1)=R_b(e)-p$, and $\lambda_1\neq 0$ if and only if $v(e)=p$.  If we blow up $y_0=Q$ to obtain the first step of the canonical resolution (see Definition \ref{canonical}), we can obtain a recursion relation as follows.
\begin{Lem} \begin{enumerate}\label{recursions}
\item If $R_b(e)\ge p$, then 
\begin{equation}\label{recursion Mild I}
\begin{split}
\xi_{I,e}=\frac{(p-1)(p-3)}{8}+\xi_{II, e_1}, \ \xi_{II,e}=\frac{(p-1)(p+1)}{8}+\xi_{I,e_1};
\end{split}
\end{equation}\label{recursion Mild II}
\item If $R_b(e)\ge p$ and $v(e)>p$, then \begin{equation}
\begin{split}
\xi_{III,e}=\frac{(p-1)(p+1)}{8}+\xi_{III, e_1}, \
\xi_{IV,e}=\frac{(p-1)(p+1)}{8}+\xi_{IV, e_1}.
\end{split}\end{equation}
\item If $R_b(e)\ge p$ and $v(e)=p$,  then 
\begin{equation}\label{recursion Wild}
\begin{split}
\xi_{III,e}=\frac{(p-1)(p+1)}{8}+\xi_{I,e_1},\  \xi_{IV,e}=\frac{(p-1)(p+1)}{8}+\xi_{II,e_1}.
\end{split}
\end{equation}
\end{enumerate}
\end{Lem}
\begin{proof}
According to the process of canonical resolution, after blowing-up we can get the two tables (TABLE 1 \& 2) below, everything then follows from the table. We remark that it is clear outside the open subset $\Spec(\cO_{b,C}[x/t])$, $B_1$ could have at worst negligible singularities.
\end{proof}

\begin{center}
\begin{table}[!htp]
\begin{tabular}{|c|c|c|c|c|c|c|}
\hline
 &  $m_0$ & $l_0$ & equation of $B_1$ on   &  equation of & $(l_0^2-l_0)/2$\\
& & &$\Spec(\cO_{b,C}[x/t])$  &  singularities &\\
\hline
I  & $p$ & $\dfrac{p-1}{2}$ & $t((x/t)^p-e_1) $ &   $t((x/t)^p-e_1)$ & $\dfrac{(p-1)(p-3)}{8}$\\
\hline
II  & $p+1$ & $\dfrac{p+1}{2}$ & $(x/t)^p-e_1$  &  $(x/t)^p-e_1$ & $\dfrac{(p-1)(p+1)}{8}$\\
\hline
III & $p+1$ & $\dfrac{p+1}{2}$ & $(x/t)((x/t)^p-e_1)$  & $(x/t)((x/t)^p-e_1)$ & $\dfrac{(p-1)(p+1)}{8}$\\
\hline
IV  & $p+2$ & $\dfrac{p+1}{2}$ & $t(x/t)((x/t)^p-e_1)$  & $t(x/t)((x/t)^p-e_1)$ & $\dfrac{(p-1)(p+1)}{8}$\\
\hline
\end{tabular}
\medskip 
\caption{table of $\lambda_1=0$.}
\end{table}
\begin{table}[!htp]
\begin{tabular}{|c|c|c|c|c|c|c|}
\hline
 &  $m_0$ & $l_0$ & equation of $B_1$ on    &  equation of  & $(l_0^2-l_0)/2$\\
& & & $\Spec(\cO_{b,C}[x/t])$ &  singularities &\\
\hline
I  & $p$ & $\dfrac{p-1}{2}$ & $t((x/t-\lambda_1^{1/p})^p-e_1) $ &   $t((x/t-\lambda_1^{1/p})^p-e_1)$ & $\dfrac{(p-1)(p-3)}{8}$ \\
\hline
II  & $p+1$ & $\dfrac{p+1}{2}$ & $(x/t-\lambda_1^{1/p})^p-e_1 $  &  $(x/t-\lambda_1^{1/p})^p-e_1$&$\dfrac{(p-1)(p+1)}{8}$ \\
\hline
III & $p+1$ & $\dfrac{p+1}{2}$ & $(x/t)((x/t-\lambda_1^{1/p})^p-e_1)$  & $(x/t-\lambda_1^{1/p})^p-e_1$ &$\dfrac{(p-1)(p+1)}{8}$\\
\hline
IV  & $p+2$ & $\dfrac{p+1}{2}$ & $t(x/t)((x/t-\lambda_1^{1/p})^p-e_1)$  & $t((x/t-\lambda_1^{1/p})^p-e_1)$ &$\dfrac{(p-1)(p+1)}{8}$\\
\hline
\end{tabular}\hfill
\medskip 
\caption{table of $\lambda_1\neq 0$.}
\end{table}
\end{center}

\begin{Lem}\label{104}\begin{enumerate}
\item The number $\xi_{*,e}$ depends on the ramification type $\Lambda_{b}(e)$ (see Definition~\ref{notation}) rather than $e$ itself.
\item If $*$ is I or II, then $\xi_{*,e}$ depends on $R_b(e)$ only.
\end{enumerate}
Since $\xi_{*,e}$ depends on the ramification type $\Lambda_b(e)$ rather than $e$, we shall also write $\xi_{*,\Lambda}$ to denote $\xi_{*,e}$ for those $e$ with $\Lambda_b(e)=\Lambda$.
\end{Lem}
\begin{proof}
By Lemma \ref{recursions} if $R_b(e)\ge p$ then $\xi_{*,e}$  is determined by 
$\xi_{*, e_1}$ and whether $\lambda_1=0$ or not. However it is clear that the ramification type of $\Lambda_b(e)$ is also determined by $\Lambda_b(e_1)$ and whether $\lambda_1=0$ or not,  so the our lemma is true if it is true for cases $R_b(e)< p$,  the latter is  clear.
\end{proof}

\begin{Lem}\label{from recursions to inequalities}  
\begin{enumerate}[a)]
\item For any $e$, we have \begin{align}
\label{100}\frac{(p-1)^2R_b(e)}{8p}-\xi_{I,e}&\ge 0;\\
\label{101}\frac{(p-1)^2R_b(e)}{8p}+\frac{p-1}{4}-\xi_{II,e}&\ge 0;
\end{align}
\item For any $e$ with  tame ramification, we have
\begin{align}
\label{102}\frac{(p-1)(p+1)R_b(e)}{8p}-\xi_{III,e}&\ge -\frac{p-1}{4p};\\
\label{103}\frac{(p-1)(p+1)R_b(e)}{8p}+\frac{p-1}{4}-\xi_{IV,e}&\ge -\frac{p-1}{4p};
\end{align}
\item If $e$ has  wild ramification and $\Lambda_b(e)=\{pj,R_b(e)\}$, then 
\begin{align}
\frac{(p-1)^2R_b(e)}{8p}-\xi_{III,e}\ge -\frac{(p-1)j}{4};\\
\frac{(p-1)^2R_b(e)}{8p}+\frac{p-1}{4}-\xi_{IV,e}\ge -\frac{(p-1)j}{4}.
\end{align}\end{enumerate}
\end{Lem}
\begin{proof}
The previous lemma reduces our statements a) and b) to Proposition~\ref{App 1} below. For c), we assume $\Lambda(e)=\{pj,R_b(e)\}$, then $e=t^{pj}(\lambda+e')$, with $\lambda\neq 0$. Hence by Lemma \ref{recursions} we have:
\begin{align*}
\frac{(p-1)^2R_b(e)}{8p}-\xi_{III,e}&=\frac{(p-1)^2R_b(e')}{8p}-\xi_{I,e'}+\frac{(p-1)^2j}{8}-\frac{(p+1)(p-1)j}{8}\\&=(\frac{(p-1)^2R_b(e')}{8p}-\xi_{I,e'})-\frac{(p-1)j}{4}\\&\ge -\frac{(p-1)j}{4},
\end{align*}
and similarly $$\frac{(p-1)^2R_b(e)}{8p}+\frac{p-1}{4}-\xi_{IV,e}\ge -\frac{(p-1)j}{4}.$$
\end{proof}
\begin{Lem}\label{formula of A}
$\alpha= \sum \limits_{b\in\mathcal{T}}(R_{b}(h)+1)+\sum \limits_{b\in\mathcal{W}}(pj_b(h))$.
\end{Lem}
\begin{proof}
By definition $$A=\sum\limits_{b\in \mathcal{T}}(R_b(h)+1)b+\sum\limits_{b\in\mathcal{W}}pj_b(h)b.$$ Taking degree we obtain our lemma.
\end{proof}
\begin{Thm}\label{evidence}
Under the assumption $g=(p-1)/2$ and $X_\eta$ being quasi-hyperelliptic, we have $\chi(\cO_X)\ge (p^2-4p-1)(q-1)/4p$.
\end{Thm}
\begin{proof}
By Equation (\ref{fine}) $$\chi(\cO_X)=\frac{(p-3)(q-1)}{2}+\frac{(p-1)(\alpha+d)}{4}-\sum \limits_{b\in\mathcal{S}} \xi_{b}.$$ 
Lemma \ref{from recursions to inequalities} and Lemma~\ref{formula of A} show that 
\begin{align*}
\frac{(p-1)d}{4}-\sum\limits_{b\in \mathcal{S}}\xi_b&\ge -\sum\limits_{b\in \mathcal{S}}\frac{(p-1)^2R_b(h)}{8p}-\sum\limits_{b\in \mathcal{T}}\frac{(p-1)(R_b(h)+1)}{4p}-\sum\limits_{b\in \mathcal{W}}\frac{(p-1)j}{4}\\&=-\sum\limits_{b\in \mathcal{S}}\frac{(p-1)^2R_b(h)}{8p}-\frac{(p-1)\alpha}{4p}.
\end{align*}
Hence \begin{align*}
\chi(\cO_X)&=\frac{(p-3)(q-1)}{2}+\frac{(p-1)(\alpha+d)}{4}-\sum \limits_{b\in\mathcal{S}} \xi_{b}\\
&\ge \frac{(p-3)(q-1)}{2}+\frac{(p-1)\alpha}{4}-\sum\limits_{b\in \mathcal{S}}\frac{(p-1)^2R_b(h)}{8p}-\frac{(p-1)\alpha}{4p}\\
&=\frac{(p^2-4p-1)(q-1)}{4p},
\end{align*} by Hurwitz's formula:
\begin{equation}
2\alpha+2(q-1)=\sum\limits_{b\in \mathcal{S}} R_b(h).
\end{equation}
\end{proof}
\begin{Cor}
Under the assumption $g=(p-1)/2$ and $X_\eta$ being quasi-hyperelliptic, the optimal bound of $\chi/c_1^2$ is $(p^2-4p-1)/4(3p^2-8p-3)$.
\end{Cor}
\begin{proof}
Since $\chi(\cO_X)\ge (p^2-4p-1)(q-1)/4p$, we see that $$\frac{\chi(\cO_X)}{K_X^2}=\frac{\chi(\cO_X)}{12\chi(\cO_X)-c_2(X)}\ge \frac{\chi(\cO_X)}{12\chi(\cO_X)+4(q-1)}\ge \frac{p^2-4p-1}{4(3p^2-8p-3)}.$$ On the other hand, Raynaud's example in Subsection \ref{Raynaud's example} gives examples whose  $\chi/c_1^2$ is equal to $(p^2-4p-1)/4(3p^2-8p-3)$.
\end{proof}
\begin{Cor}\label{last}
We have $\kappa_5=1/32$.
\end{Cor}
\begin{proof}
When $g=(p-1)/2$, $X_\eta$ is automatically hyperelliptic, hence the best bound of $\chi/c_1^2$ is $1/32$ for these surfaces. Combining this with Corollary~\ref{main 5,4}, we obtain $\kappa_5=1/32$. 
\end{proof}

\section{Appendix}
Assume  $\mathrm{char}(k)\neq 2$, $a,b\in\{0,1\}$, and $m,n\in \mathbb{N}_{+}$ are two  numbers co-prime to each other. Let $S:=\Spec(k[\![x,y,t]\!]_{(x,y,t)}/(y^2-x^at^b(x^m-t^n)))$ and  $f:\widetilde
{S}\to S$ be an arbitrary desingularization, we define $\xi(a,b,m,n):=\dim_k R^1f_*\cO_{\widetilde{S}}.$

\begin{Prop}\label{App 1}
If $2\nmid m$, then 
$$\xi(a,b,m,n)\le \frac{(m-1)^2(n-1)}{8m}+\frac{(m-1)n}{4m}a+\frac{m-1}{4}b.$$
\end{Prop}

First we point out an algorithm of calculating of $\xi(a,b,m,n)$. 
\begin{Lem}\label{App 2}
\begin{enumerate}
\item If $m=1$ or $n=1$, $\xi(a,b,m,n)=0$;
\item If $m>n>1$,  then 
$$\xi(a,b,m,n)=\left\{\begin{array}{c}
\xi(0,b,m-n,n)+(a+b+n)(a+b+n-2)/8, \\
\text{if} \ 2\mid a+b+n;\\
\xi(1,b,m-n,n)+(a+b+n-1)(a+b+n-3)/8,\\ 
\text{if} \ 2\nmid a+b+n.
\end{array}\right.$$
\item If $n>m>1$, then
$$\xi(a,b,m,n)=\left\{\begin{array}{c}
\xi(a,0,m,n-m)+(a+b+m)(a+b+m-2)/8,\\
 \text{if} \ 2\mid a+b+m;\\
\xi(a,1,m,n-m)+(a+b+m-1)(a+b+m-3)/8,\\ \text{if} \ 2\nmid a+b+m.
\end{array}\right.$$
\end{enumerate}
\end{Lem}
\begin{proof}
In fact $S$ is obtained as a flat double cover of $Y:=\Spec(k[x,t])$ with branch divisor $B=\mathrm{div}(x^at^b(x^m-t^n))$. Our lemma follows from the process of the canonical resolution(see Definition~\ref{canonical}). 
\end{proof}
\begin{Lem}\label{App 3}
Proposition~\ref{App 1} holds if it holds for all $n<m$.
\end{Lem}
\begin{proof}
Let $n=m+n'$. If $2\mid a+b+m$, then by Lemma~\ref{App 2} we have
\begin{align*}
&\frac{(m-1)^2(n-1)}{8m}+\frac{(m-1)n}{4m}a+\frac{m-1}{4}b-\xi(a,b,m,n)\\
\ge &\frac{(m-1)^2(n'-1)}{8m}+\frac{(m-1)n'}{4m}a-\xi(a,0,m,n').
\end{align*}
If $2\nmid a+b+m$, then we also have 
\begin{align*}
&\frac{(m-1)^2(n-1)}{8m}+\frac{(m-1)n}{4m}a+\frac{m-1}{4}b-\xi(a,b,m,n)\\
\ge &\frac{(m-1)^2(n'-1)}{8m}+\frac{(m-1)n'}{4m}a+\frac{m-1}{4}-\xi(a,1,m,n').
\end{align*} So it is sufficient to prove the inequality  for pair $(m,n')$.
\end{proof}

\begin{proof}[Proof of Proposition~\ref{App 1}]
We shall proceed by induction on $m$. When $m=1$, the statement holds trivially. Assume our proposition holds for odd numbers smaller than $m$, we need to show it also holds for $m$. By Lemma~\ref{App 3}, we can assume $n<m$. 

If $2\nmid n$, then \begin{align*}
\xi(a,b,m,n)=\xi(b,a,n,m)&\le \frac{(n-1)^2(m-1)}{8n}+\frac{(n-1)m}{4n}b+\frac{n-1}{4}a\\&\le \frac{(m-1)^2(n-1)}{8m}+\frac{m-1}{4}b+\frac{(m-1)n}{4m}a.
\end{align*}

If $2\mid n$, let $m=n+m'$, then by Lemma~\ref{App 2}, we have 
\begin{align*}
\xi(a,0,m,n)&=\xi(a,0,m',n)+\frac{n(n-2)}{8}\\&\le\frac{(m'-1)^2(n-1)}{8m'}+\frac{(m'-1)n}{m'}a+\frac{n(n-2)}{8}\\
&<\frac{(m-1)^2(n-1)}{8m}+\frac{(m-1)n}{m}a
\end{align*}\begin{align*}
\xi(0,1,m,n)&=\xi(1,1,m',n)+\frac{n(n-2)}{8}\\&\le \frac{(m'-1)^2(n-1)}{8m'}+\frac{(m'-1)n}{4m'}+\frac{m'-1}{4}+\frac{n(n-2)}{8}\\
&<\frac{(m-1)^2(n-1)}{8m}+\frac{m-1}{4}
\end{align*}\begin{align*}
\xi(1,1,m,n)&=\xi(0,1,m',n)+\frac{n(n+2)}{8}\\&\le\frac{(m'-1)^2(n-1)}{8m'}+\frac{m'-1}{4}+\frac{n(n+2)}{8}\\
&\le\frac{(m-1)^2(n-1)}{8m}+\frac{(m-1)n}{m}+\frac{m-1}{4}
\end{align*}
Here we note that $$\frac{(m-1)^2(n-1)}{8m}-\frac{(m'-1)^2(n-1)}{8m'}=\frac{n(n-1)}{8}-\frac{n(n-1)}{8mm'},$$ and the last equality holds only if $n=m-1$.
\end{proof}

\section*{\bf Acknowledgements} I would like to thank  Prof. Jinxing Cai  for suggesting this problem. I would also like to thank  Prof. Qing Liu  for a lot of discussions and useful suggestions to  improve this paper. Finally I would like to thank 
Universit\'e  de Bordeaux for  hospitality and China Scholarship Council for
financial support.

\end{document}